\documentclass[12pt]{article}

\usepackage[vmargin=2.5cm,hmargin=2.2cm]{geometry}

\usepackage{lineno,hyperref} 
\usepackage{moreverb}
\usepackage{amsmath,amsthm}
\usepackage{color,xcolor}
\usepackage{amsfonts}

\usepackage{mathrsfs}
\usepackage{amssymb}
\usepackage{hyperref}
\usepackage{graphicx}
\usepackage{multicol}
\usepackage{bm,enumerate}

\newtheorem{theorem}{Theorem}[section]
\newtheorem{definition}[theorem]{Definition}  
\newtheorem{lemma}[theorem]{Lemma}
\newtheorem{corollary}[theorem]{Corollary}

\newtheorem{remark}[theorem]{Remark}
\newtheorem{assumption}{Assumption}
\newtheorem{condition}{Condition}

\renewcommand{\geq}{\geqslant}
\renewcommand{\leq}{\leqslant}

\newcommand{\dd}{\operatorname{d}\! }
\newcommand{\dt}{\operatorname{d}\! t}
\newcommand{\ds}{\operatorname{d}\! s}

\newcommand{\dw}{\operatorname{d}\! W}

\newcommand{\dtp}{\dt\otimes\mathrm{d}\mathbb{P}\text{-a.s.}}

\newcommand{\interval}{[-\bar{\epsilon},\bar{\epsilon}]}

\newcommand{\nn}{\nonumber}

\begin{document}

\title{Multidimensional indefinite stochastic Riccati equations \\
and zero-sum stochastic linear-quadratic differential games\\
with non-Markovian regime switching
}

\author{Panpan Zhang\thanks{School of Control Science and Engineering, Shandong University, Jinan 250061, Shandong, China. Email: \url{zhangpanpan@mail.sdu.edu.cn}}
\and Zuo Quan Xu\thanks{Department of Applied Mathematics, The Hong Kong Polytechnic University, Kowloon, Hong Kong, China. \url{maxu@polyu.edu.hk}}
}
\maketitle 

\begin{abstract}
This paper is concerned with zero-sum stochastic linear-quadratic differential games in a regime switching model. The coefficients of the games depend on the underlying noises, so it is a non-Markovian regime switching model. Based on the solutions of a new kind of multidimensional indefinite stochastic Riccati equation (SRE) and a multidimensional linear backward stochastic differential equation (BSDE) with unbounded coefficients, we provide closed-loop optimal feedback control-strategy pairs for the two players. The main contribution of this paper, which is of great importance in its own right from the BSDE theory point of view, is to prove the existence and uniqueness of the solution to the new kind of SRE. Notably, the first component of the solution (as a process) is capable of taking positive and negative values simultaneously. For homogeneous systems, we obtain the optimal feedback control-strategy pairs under general closed convex cone control constraints. Finally, these results are applied to portfolio selection games with full or partial no-shorting constraint in a regime switching market with random coefficients.\bigskip
\end{abstract}

\section{Introduction}

Differential games explore the decision-making processes of two or more individuals (referred to as players) who simultaneously make choices while considering the trade-offs with their counterparts within some continuous-time dynamic systems. A zero-sum game is a bilateral game with a singular objective, perceived as a gain for one participant and an equivalent loss for the other.
The study of zero-sum differential games can be traced back to the pioneering work of Isaacs \cite{Isaacs} who studied the game within a deterministic framework.
Fleming and Souganidis \cite{Fleming} initiated the study of zero-sum differential games within random frameworks, which are now called stochastic differential games.
Elliott and Kalton \cite{Elliott} introduced the concept of upper and lower value functions and proved that the two value functions are the unique viscosity solutions to the associated Hamilton-Jacobi-Bellman-Isaacs equations.
Recently, Buckdahn and Li \cite{Buckdahn Li} studied a zero-sum stochastic differential game with recursive utilities by backward stochastic differential equation (BSDE, for short) approach. Yu \cite{Yu}
delved into a zero-sum stochastic linear-quadratic (LQ) differential game, leveraging the advantageous structure of the LQ system to derive an optimal feedback control-strategy pair through the solution of the associated Riccati equation.

In the aforementioned studies, all the market coefficients are presumed to be deterministic, rendering the Riccati equations as one-dimensional ordinary differential equations (ODEs).
Moon \cite{Moon} extended Yu's \cite{Yu} result to Markov jump systems, wherein the coefficients are deterministic functions of both time and regime. In his model, the Riccati equation becomes a multidimensional ODE, a complexity introduced by the existence of multiple regimes.
Lv \cite{Lv} investigated an infinite horizon zero-sum stochastic differential game within a regime-switching model.

However, assuming all the market coefficients are deterministic functions of time and regime may be too restrictive. In practice, market parameters such as interest rates, stock appreciation and volatility rates, are influenced by various factors, including politics, economic growth rates and so on. Therefore, it is necessary to allow the market parameters to depend not only on the regime (which reflects the macroeconomic status) but also on other random factors (which reflect some micro noises).
With this consideration, this paper studies zero-sum stochastic LQ differential (SLQD) games with regime switching and random coefficients. Since the coefficients depend on both the regime and the underlying noises, we are dealing with a non-Markovian regime-switching model so that the ODE approach fails.

It is well-known that the closed-form representation of the optimal control for a stochastic LQ control problem with deterministic coefficients is closely related to the solvability of the corresponding Riccati equation. This equation, which is an ODE, typically exhibits growth that exceeds linearity. When dealing with a non-Markovian model, the Riccati equation becomes a BSDE, known as a stochastic Riccati equation (SRE), which features a non-linear growth generator.

Bismut \cite{Bismut} was a pioneer in the study of LQ control problems with random coefficients, successfully addressing some specific cases --- particularly those with a linear generator. However, due to the high degree of nonlinearity, the solvability of general SREs has remained a challenging and long-standing issue.
Tang \cite{Tang SMP} was the first to establish the existence and uniqueness result for general SREs by employing stochastic Hamiltonian systems. In his work, the random control weighting matrix was required to be uniformly positive definite. For those interested in the LQ control problem with random coefficients, further reading can be found in the works of Kohlmann and Tang \cite{Kohlmann Tang}, Hu and Tang \cite{Hu Tang}, Tang \cite{Tang DPP}, Hu and Zhou \cite{Hu Zhou}, and Sun \cite{Sun}.

To study zero-sum SLQD games within non-Markovian regime switching frameworks, it is inevitable to encounter indefinite SREs, where the state or control weighting matrices may possess zero and negative eigenvalues. Indefinite SREs are also prevalent in stochastic LQ optimal control problems with random coefficients. Generally, the solvability of indefinite SREs presents an exceptionally challenging and long-standing issue.
Existence and/or uniqueness results for indefinite SREs have been established in some special, yet significant, cases, as documented in the works of Hu and Zhou \cite{Hu Zhou 2}, Qian and Zhou \cite{Qian Zhou}, Du \cite{Du}, Hu, Shi and Xu \cite{Xu AAP}, among others.

However, it is crucial to recognize that the indefinite SREs arising from zero-sum SLQD games and those from optimal stochastic control problems are fundamentally different.
This distinction stems from the fact that in zero-sum games, the objectives of the two players are inherently opposed, leading to control weighting matrices that generally exhibit opposite signs (thus being indefinite). In contrast, participants in a control problem, even if there are multiple, typically aim in the same direction (such as minimizing nonnegative quadratic cost functionals). Consequently, the weighting matrices in control problems are usually positive semi-definite to ensure that the value function is lower bounded.

The study of optimal control problems within non-Markovian regime-switching frameworks has only recently commenced. Hu, Shi, and Xu \cite{Xu AAP} were the pioneers in this area, initially tackling a homogenous stochastic LQ optimal control problem that incorporated closed convex cone constraint on the control variable. They applied their findings to a continuous-time mean-variance portfolio selection problem.
Subsequently, they \cite{Xu arXiv1} expanded their model to address inhomogeneous systems, successfully solving a novel class of multidimensional linear BSDEs with unbounded coefficients. In a recent work, they \cite{Xu arXiv2} delved into finite-time optimal consumption-investment problems featuring power, logarithmic, and exponential utilities within a regime-switching market characterized by random coefficients. Wen et al. \cite{WLXZ23} built upon the vector-valued case presented in \cite{Xu AAP} to explore the matrix-valued scenario, broadening the scope of applications and theoretical understanding.
Additionally, Moon \cite{Moon2} investigated zero-sum stochastic differential games where the diffusion term does not depend on controls.
This body of research collectively advances the understanding of complex control problems in financial mathematics and beyond.

In this paper, we explore zero-sum SLQD games for systems with non-Markovian regime switching. The coefficients' randomness originates from two sources: the Brownian motion, which represents the underlying noises, and the Markov chain, which accounts for the regime switching. Our work extends the control theory presented in previous studies Hu, Shi and Xu \cite{Xu AAP,Xu arXiv1} to the zero-sum game context, but there are notable distinctions between our approach and that of the referenced works.
Specifically, the weighting matrices in the prior studies are described as \emph{weak indefinite}, meaning they are positive semi-definite (their eigenvalues can be \emph{zero or positive}, but never negative). In contrast, the weighting matrices in our study are \emph{strong indefinite}, where the eigenvalues take \emph{both positive and negative values} simultaneously. This difference has significant implications: the solutions to the SREs in the previous research all have nonnegative first components, whereas in our case, the first components of the SRE solutions can take positive and negative values concurrently. This presents unique challenges and necessitates novel approaches to solve the SREs and derive the optimal strategies for zero-sum SLQD games.

We begin by defining admissible feedback controls and admissible feedback strategies for our games within the context of non-Markovian regime switching. For our investigation, we introduce a novel type of multidimensional indefinite SRE, the solution of which may simultaneously assume positive and negative values. Leveraging a stability result for BSDEs by Cvitanic and Zhang \cite{Zhang book}, along with a multidimensional comparison theorem by Hu and Peng \cite{Hu Peng}, we are able to establish the existence of the solution to the multidimensional indefinite SRE through monotone approximation.

However, our attempts to establish a uniqueness result directly through BSDE methodologies, such as the log transformation technique successfully applied in the previous study \cite{Xu AAP} for certain weak indefinite SREs, have not been fruitful. Specifically, we are unable to identify a transformation for our indefinite SRE that ensures the quadratic term in the generator is monotone --- a crucial step in the aforementioned study. This challenge arises because the objectives of the two players in our problem are inherently opposed, leading to an inevitable loss of monotonicity.
Instead, we resort to a verification theorem to establish uniqueness. By employing the technique of completing squares, we derive optimal feedback control-strategy pairs for the two players based on the solutions of the indefinite SRE and a multidimensional linear BSDE with unbounded coefficients, thereby implying uniqueness. It is important to note that this approach relies on the specific structure of the SRE, which is directly linked to a game scenario. The pursuit of proving uniqueness directly through a BSDE approach remains an open and significant area for future research.

Furthermore, we extend our analysis to include games for homogeneous systems. In such scenarios, we are able to integrate closed convex cone control constraint into the game framework and provide the corresponding optimal feedback control-strategy pairs for the players.

As a practical application of our theoretical results, we examine portfolio selection games that feature various short-selling prohibition constraints within a non-Markovian regime-switching market. This allows us to explore how the absence of short-selling opportunities, a common regulatory or self-imposed restriction in financial markets, affects the optimal strategies for investors operating under regime-switching conditions that introduce additional layers of complexity and uncertainty.
Our study not only advances the theoretical understanding of zero-sum SLQD games with non-Markovian regime switching but also offers valuable insights into the real-world implications of these models, particularly in the context of financial decision-making under regulatory constraints and market uncertainties.

The remainder of this paper is organized as follows. In Section \ref{PF}, we formulate a zero-sum SLQD game for inhomogeneous systems with non-Markovian regime switching.
In Section \ref{Solution}, we give the optimal feedback control-strategy pairs for the LQ game and prove the solvability of the related multidimensional indefinite SRE.
Section \ref{Game} is concerned about constrained LQ game for homogeneous systems.
In Section \ref{Application}, we apply the general results to solve several portfolio selection problems with possible short-selling prohibition constraints. Finally, Section \ref{Conclusion} concludes the paper.

\section{Problem Formulation}\label{PF}

Let $(\Omega, \mathcal{F}, \{\mathcal{F}_{t}\}_{0\leq t\leq T}, \mathbb{P})$
be a fixed complete probability space
where $\mathcal{F}=\mathcal{F}_{T}$ and $T>0$ is a fixed time horizon. Let $\mathbb{E}$ be the expectation with respect to (w.r.t.) $\mathbb{P}$.
In this space, define a standard $n$-dimensional Brownian motion $W(t) = (W_1(t),\ldots,W_n(t))^{\top}$, $t\in[0,T]$,
and an independent continuous-time stationary Markov chain $\alpha_t$, $t\in[0,T]$, valued in a finite state space $\mathcal{M} = \{1, 2,\ldots, l\}$ with $l \geq 1$. The superscript $\top$ denotes the transpose of vectors or matrices.
The Markov chain $\alpha_\cdot$ has a generator
$Q= (q_{ij })_{l\times l} $ with $q_{ij}\geq 0$ for $i \neq j$ and $\sum_{j=1}^{l} q_{ij} = 0$ for every $i \in \mathcal{M}$.
We assume $\mathcal{F}_{t}=\sigma\{ W(s), \alpha_s: 0\leq s\leq t\}\bigvee \mathcal{N}$, where $\mathcal{N}$ is the totality of all the $\mathbb{P}$-null sets of $\mathcal{F}$ and denote $\mathcal{F}^{W}_{t}=\sigma\{ W(s): 0\leq s\leq t\}\bigvee \mathcal{N}$.

We denote by $\mathbb{R}^{n}$ the $n$-dimensional real-valued Euclidean space with the Euclidean norm $|\cdot|$, by $\mathbb{R}_{>0}$ the set of all positive real numbers,
by $\mathbb{R}^{n\times m}$ the set of all $n\times m$ real-valued matrices,
and by $\mathbb{S}^{n} $ the set of all $n\times n$ real-valued symmetric matrices.
We use $I_{n}$ to denote the $n$-dimensional identity matrix and $\bm{0}$ denote 0 vectors or matrices with proper size which may vary from line to line. We define $x^{+}=\max\{x,0\}$, $x^{-}=\max\{-x,0\}$, $x \vee y=\max\{x,y\}$ and
$x \wedge y=\min\{x, y\}$ for $x,y\in\mathbb{R}$.
For $S\in \mathbb{S}^n$, $c\in \mathbb{R}$,
write $S\geq cI_{n}$ if $ y^{\top} S y\geq c|y|^2$ holds for any $y\in \mathbb{R}^{n}$,
and define $S\leq cI_{n}$ similarly.

We use the following spaces throughout the paper:\medskip\\ 
\begin{tabular}{rll}
$L^{\infty}_{\mathcal{F}_T}(\mathbb{R}^{n}):\!\!\!\!$&the set of all $\mathbb{R}^{n}$-valued $\mathcal{F}_T$-measurable essentially bounded random variables; \medskip\\

$L^2_{\mathcal{F}_T}(\mathbb{R}^{n}) :\!\!\!\!$& the set of all $\mathbb{R}^{n}$-valued $\mathcal{F}_T$-measurable random variables $\xi$ such\\
&that $\mathbb{E}\big[ |\xi|^2 \big]<\infty$; \medskip\\

$L^{\infty}_{\mathcal{F}}(0,T;\mathbb{R}^{n}) :\!\!\!\!$& the set of all $\mathcal{F}_t$-adapted essentially bounded processes $v:[0,T]\times \Omega\rightarrow \mathbb{R}^{n}$; \medskip\\

$L^2_{\mathcal{F}}(0,T;\mathbb{R}^{n}) :\!\!\!\!$& the set of all $\mathcal{F}_t$-adapted processes
$v:[0,T]\times \Omega\rightarrow \mathbb{R}^{n}$ such\\
&that $\mathbb{E}\big[\int_0^T |v(t)|^2\dt \big]<\infty$; \medskip\\

$L^{2,\text{loc}}_{\mathcal{F}}(0,T;\mathbb{R}^{n}) :\!\!\!\!$& the set of all $\mathcal{F}_t$-adapted processes
$v:[0,T]\times \Omega\rightarrow \mathbb{R}^{n}$ such\\
&that $\mathbb{P}\big(\int_0^T |v(t)|^2\dt <\infty \big)=1$; \medskip\\

$L^{2}_{\mathcal{F}}(C(0,T);\mathbb{R}^{n}):\!\!\!\!$& the set of all $\mathcal{F}_t$-adapted processes $v:[0,T]\times \Omega\rightarrow \mathbb{R}^{n}$ with continuous \\
&sample paths and $\mathbb{E}\Big[\sup \limits_{t\in[0,T]} |v(t)|^2 \Big]<\infty$.\medskip
\end{tabular}

\noindent These definitions are generalized in the obvious way to the cases that $\mathcal{F}$ is replaced by $\mathcal{F}^W$
and $\mathbb{R}^{n}$ by $\mathbb{R}$, $\mathbb{R}^{n\times m}$ or $\mathbb{S}^{n}$.
For notation simplicity, all the estimates between stochastic processes (resp. random variables) hold in the sense that $\dtp$ (resp. $\mathrm{d}\mathbb{P}$-a.s.).
All the processes unless otherwise stated are stochastic, so we omit the argument $\omega\in \Omega$.
Furthermore, some arguments, particularly those in integrands, such as $ t$, $\alpha_{t}$,and $i$ may be suppressed in circumstances when no confusion occurs.

This paper studies zero-sum SLQD games, where the controlled state process satisfies a scalar-valued inhomogeneous linear stochastic differential equation (SDE):
\begin{equation}\label{201}
\left\{
\begin{aligned}
\dd X(t)=\,&\big[ A(t,\alpha_t)X(t)+B_1(t,\alpha_t)^{\top}u_1(t) +B_2(t,\alpha_t)^{\top}u_2(t) +b(t,\alpha_t) \big]\dt \\
&+ \big[ C(t,\alpha_t)X(t)+D_1(t,\alpha_t)u_1(t) +D_2(t,\alpha_t)u_2(t)+\sigma(t,\alpha_t) \big]^{\top}\dw(t),\\
X(0)=\,&x\in\mathbb{R},\quad \alpha_0=i_0\in \mathcal{M}.
\end{aligned}
\right.
\end{equation}
Here $X(\cdot)$ denotes the state process, $u_1(\cdot)$ and $u_2(\cdot)$ are the control processes, and $(x,i_{0})$ is the initial state.
The objective functional is of quadratic form\footnote{Of course, one can introduce inhomogeneous terms in the objective functional as well, but the arguments are similar.}:
\begin{align}\label{202}
J_{x,i_0}( u_1,u_2)=\mathbb{E}\bigg[ \int_0^T\Big( & K(t,\alpha_t)X(t)^2+u_1(t)^{\top}R_{11}(t,\alpha_t)u_1(t)+2u_1(t)^{\top}R_{12}(t,\alpha_t)u_2(t)\\
&\quad+u_2(t)^{\top}R_{22}(t,\alpha_t)u_2(t)\Big)\dt+ G(\alpha_T)X(T)^2 \bigg].\nn
\end{align}
This functional can be regarded as the payoff for Player 1 and the cost for Player 2.
Player 1 aims to maximize \eqref{202}, whereas Player 2 aims to maximize its opposite, thus being zero sum. The coefficients are assumed to be bounded, i.e., for all $i\in \mathcal{M}$, $k\in\{1,2\}$, we have
\begin{gather*}
A(\cdot,\cdot, i),~b(\cdot,\cdot, i)\in L^{\infty}_{\mathcal{F}^W}(0,T;\mathbb{R}),
~~C(\cdot,\cdot, i), ~\sigma(\cdot,\cdot, i)\in L_{\mathcal{F}^W}^{\infty}(0,T;\mathbb{R}^{n}),\\
B_k(\cdot,\cdot, i)\in L_{\mathcal{F}^W}^{\infty}(0,T;\mathbb{R}^{m_k}),~~
D_k(\cdot,\cdot, i)\in L_{\mathcal{F}^W}^{\infty}(0,T;\mathbb{R}^{n\times m_k}), \\
K(\cdot,\cdot, i)\in L^{\infty}_{\mathcal{F}^W}(0,T;\mathbb{R}),~~ G(\cdot,i)\in L^{\infty}_{\mathcal{F}^W_T}(\mathbb{R}), \\
R_{kk}(\cdot,\cdot, i)\in L^{\infty}_{\mathcal{F}^W}(0,T;\mathbb{S}^{ m_k}),
~~R_{12}(\cdot,\cdot, i)\in L^{\infty}_{\mathcal{F}^W}(0,T;\mathbb{R}^{ m_1\times m_2}).
\end{gather*}
Note that, given a regime $i\in\mathcal{M}$, the coefficients still depend on the Brownian motion, thus it is a non-Markovian regime switching model.

For $k \in\{1, 2\}$, let $u_k(\cdot)$ denote the control process of Player $k$, chosen from the admissible control set
$\mathcal{U}_k= L^2_{\mathcal{F}}(0,T;\mathbb{R}^{m_k})$.
Clearly, for any admissible control pair $(u_1, u_2)\in \mathcal{U}_1\times \mathcal{U}_2$, there exists a unique strong solution
$X(\cdot) \in L^{2}_{\mathcal{F}}(C(0,T);\mathbb{R})$ to SDE \eqref{201}, called the corresponding admissible state process. Furthermore, due to the boundedness of coefficients, it is easily seen that
$-\infty<J_{x,i_0}( u_1,u_2)<\infty.$ So the objective functional \eqref{202} is well-defined for any admissible control pair in $\mathcal{U}_1\times \mathcal{U}_2$.

Besides the admissible control sets, we also need to define the admissible strategies for the two players.
Our definition of non-anticipative strategies (see Elliott and Kalton \cite{Elliott})
is adopted from Buckdahn and Li \cite[Definition 3.2]{Buckdahn Li}, Yu \cite{Yu} and Lv \cite{Lv}.

\begin{definition} \label{def:admissible strategy}
An admissible strategy for Player $1$ is a mapping $\beta_1:\mathcal{U}_2\rightarrow\mathcal{U}_1$
such that for any $\mathcal{F}_t$-stopping time $\tau$ : $\Omega\rightarrow [0, T ]$ and any two controls $u_2$, $\overline{u}_2\in\mathcal{U}_2$ with $u_2= \overline{u}_2$ on $[0, \tau]$, it holds that $\beta_1(u_2) = \beta_1(\overline{u}_2)$ on $[0, \tau]$. The set of all admissible strategies for Player 1 is denoted by $\mathcal{A}_1$.
Admissible strategies $\beta_2:\mathcal{U}_1\rightarrow\mathcal{U}_2$
and the set $\mathcal{A}_2$ of them for Player 2 are defined similarly.
\end{definition}
For $(x,i_{0})\in \mathbb{R}\times \mathcal{M}$,
we define \textit{Player 1's value} and \textit{Player 2's value} as
\begin{align*}
V_1(x,i_0)\triangleq\inf \limits_{\beta_2 \in \mathcal{A}_2} \sup \limits_{u_1\in \mathcal{U}_1} J_{x,i_0}\big( u_1,\beta_2( u_1)), \\
V_2(x,i_0) \triangleq \sup \limits_{\beta_1\in \mathcal{A}_1} \inf \limits_{u_2 \in \mathcal{U}_2} J_{x,i_0}\big( \beta_1(u_2), u_2 ).
\end{align*}
When $V_1(x,i_0)$ (resp. $V_2(x,i_0)$) is finite, the zero-sum SLQD game for Player 1 (resp. Player 2) is to find an admissible pair $(u_1^{*}, \beta_2^{*})\in\mathcal{U}_1\times \mathcal{A}_2 $ (resp., $(u_2^{*}, \beta_1^{*})\in\mathcal{U}_2\times \mathcal{A}_1$)
such that
$J_{x,i_0}( u^{*}_1,\beta^{*}_2( u^{*}_1))
= V_1(x,i_0)$ (resp., $J_{x,i_0}( \beta^{*}_1( u^{*}_2),u^{*}_2) = V_2(x,i_0)$),
in which case $(u_1^{*}, \beta_2^{*})$ (resp., $(u_2^{*}, \beta_1^{*})$) is called an optimal control-strategy pair for Player 1's value (resp. Player 2's value).
For simplicity, we call this game the LQ game \eqref{201}-\eqref{202}.
If the two players' values are equal, this common value, denoted by $V(x,i_0)$, is called the value of the game.

Similar to general LQ control problems, we expect that the optimal control-strategy pair is in a feedback form.
In our setting, the two players can observe not only the time, noise, current values of the state and regime, but also current value of the other's control.
We now give the definitions of feedback controls and feedback strategies for the LQ game \eqref{201}-\eqref{202}.

\begin{definition}\label{1'control}
An admissible feedback control for Player 1 is a measurable mapping
$\pi : [0, T ] \times \mathcal{M} \times\mathbb{R} \rightarrow \mathbb{R}^{m_1} $
such that
\begin{enumerate}[(a)] 
\item for each $(i, x)\in \mathcal{M} \times \mathbb{R}$, $\pi(\cdot,i,x) $ is an $\mathcal{F}^{W}_t$-adapted process;
\item for each $u_2(\cdot) \in \mathcal{U}_2$, there exists a unique solution $X(\cdot)$ to the following SDE:
\begin{equation}\label{203}
\left\{
\begin{aligned}
\dd X(t)=\,&\big[ AX(t)+B_1^{\top}\pi(t,\alpha_t,X(t)) +B_2^{\top}u_2(t)+b \big]\dt \\
&+ \big[ CX(t)+D_1\pi(t,\alpha_t,X(t)) +D_2 u_2(t) +\sigma \big]^{\top}\dw(t),\\
X(0)=\,&x\in\mathbb{R},\quad \alpha_0=i_0\in \mathcal{M},
\end{aligned}
\right.
\end{equation}
and $\pi(\cdot,\alpha_{\cdot},X(\cdot))\in \mathcal{U}_1$.
\end{enumerate}
\end{definition} 

\begin{definition}\label{2'strategy}
An admissible feedback strategy for Player 2 is a measurable mapping
$\Pi: [0, T ]\times \mathcal{M} \times \mathbb{R} \times \mathbb{R}^{m_1}\rightarrow \mathbb{R}^{m_2}$ such that
\begin{enumerate}[(a)]
\item for each $(i, x, u)\in \mathcal{M} \times \mathbb{R} \times \mathbb{R}^{m_1}$, $\Pi(\cdot,i,x,u) $ is an $\mathcal{F}^{W}_t$-adapted process;
\item for each $u_1(\cdot) \in \mathcal{U}_1$, there exists a unique solution $X(\cdot)$ to the following SDE:
\begin{equation}\label{204}
\left\{
\begin{aligned}
\dd X(t)&=\big[ AX(t)+B_1^{\top}u_1(t) +B_2^{\top}\Pi(t,\alpha_t,X(t),u_1(t))+b \big]\dt \\
&\quad\;+ \big[ CX(t)+D_1u_1(t) +D_2\Pi(t,\alpha_t,X(t),u_1(t))+\sigma \big]^{\top}\dw(t),\\
X(0)&=x\in\mathbb{R},\quad \alpha_0=i_0\in \mathcal{M},
\end{aligned}
\right.
\end{equation}
and $\beta_2 : u_1(\cdot) \mapsto \Pi(\cdot,\alpha_{\cdot},X(\cdot),u_1(\cdot))\in \mathcal{A}_2$.
\end{enumerate}
\end{definition}

\begin{definition}
Let $\pi$ be an admissible feedback control for Player 1 and $\Pi$ be an admissible feedback strategy for Player 2. The pair $(\pi,\Pi)$ is called an optimal feedback control-strategy pair for Player 1's value if the pair $(u_1, \beta_2)$ is optimal, that is,
$J_{x,i_0}( u_1,\beta_2( u_1))= V_1(x,i_0)$,
where $(u_1, \beta_2)$ is defined by
$u_1(\cdot)=\pi(\cdot,\alpha_{\cdot},X(\cdot))$, $\beta_2 : u_1(\cdot) \mapsto \Pi(\cdot,\alpha_{\cdot},X(\cdot),u_1(\cdot))$, and $X(\cdot)$ is the solution to the following SDE:
\begin{equation*}
\left\{
\begin{aligned}
\dd X(t)=\,&\big[ AX(t)+B_1^{\top}\pi(t,\alpha_t,X(t)) +B_2^{\top}\Pi(t,\alpha_t,X(t),u_1(t))+b \big]\dt \\
&+ \big[ CX(t)+D_1\pi(t,\alpha_t,X(t))+D_2\Pi(t,\alpha_t,X(t),u_1(t))+\sigma \big]^{\top}\dw(t),\\
X(0)=\,&x\in\mathbb{R},\quad \alpha_0=i_0\in \mathcal{M}.
\end{aligned}
\right.
\end{equation*}
\end{definition}

The admissible feedback control for Player 2, the admissible feedback strategy for Player 1,
and the optimal feedback control-strategy pair for Player 2's value are defined similarly. For simplicity, we directly call $u_1(\cdot)=\pi(\cdot,\alpha_{\cdot},X(\cdot))$ defined in Definition \ref{1'control} the feedback control for Player 1, and $\beta_2 : u_1(\cdot) \mapsto \Pi(\cdot,\alpha_{\cdot},X(\cdot),u_1(\cdot))$ defined in Definition \ref{2'strategy} the feedback strategy for Player 2, respectively.

Thanks to the boundedness of coefficients, there exists positive constants
$c_1$, $\bar{c}_2$, $c_3$, $\overline{K}$, $\overline{G}$, such that
\begin{gather}
|K|\leq\overline{K},~~ |G|\leq\overline{G},~~ 2A+C^{\top}C+\max\limits_{1\leq i,j\leq l} q_{ij}\leq c_1,~~
D_1^{\top}D_1\leq \bar{c}_2I_{m_1}, ~~ D_2^{\top}D_2\leq \bar{c}_2I_{m_2},\nn\\
\label{bounds on coefficients}
\big[B_1+D_1^{\top}C\big]^{\top}\big[B_1+D_1^{\top}C\big] \vee \big[B_2+D_2^{\top}C\big]^{\top}\big[B_2+D_2^{\top}C\big] \leq c_3.
\end{gather}
We define two positive constants
$$
\epsilon \triangleq 8 c_3 (e^{c_1lT}-1) [\overline{K}(e^{c_1lT}-1)+\overline{G}c_1le^{c_1lT} ] / (c_1 l)^2,\quad
\bar{\epsilon}\triangleq c_1 l \epsilon / [4c_3(e^{c_1lT}-1)].
$$
These notations will be used throughout this paper.

To ensure our problem is well-posed, we assume from now on the follow three assumptions hold without claim.
\begin{assumption}\label{ass1}
There exists a positive constant $\underline{c}_2$ such that $\underline{c}_2I_{m_1}\leq D_1^{\top}D_1$ and $\underline{c}_2I_{m_2}\leq D_2^{\top}D_2$ for all $i\in\mathcal{M}$.
\end{assumption}

\begin{assumption}\label{ass2}
It holds that $R_{11}\leq -(\epsilon+\bar{\epsilon}\bar{c}_2 )I_{m_1}$ and
$R_{22}\geq ( \epsilon+\bar{\epsilon}\bar{c}_2 )I_{m_2}$ for all $i\in\mathcal{M}$.
\end{assumption}

\begin{assumption}\label{ass3}
It holds that
$2\bar{c}_2 <\epsilon$.
\end{assumption}

\begin{remark}
In fact, given Assumptions \ref{ass1} and \ref{ass2}, Assumption \ref{ass3} is redundant. To see this, suppose the LQ game \eqref{201}-\eqref{202} satisfies Assumptions \ref{ass1} and \ref{ass2}, then the game is clearly equivalent to the game under the system \eqref{201} with the objective functional
\begin{equation*}
\begin{aligned}
\tilde{c}J_{x,i_0}( u_1,u_2)\triangleq\mathbb{E}\bigg[ \int_0^T \Big(& \tilde{c}KX(t)^2
+u_1(t)^{\top}\tilde{c}R_{11}u_1(t)+2u_1(t)^{\top}\tilde{c}R_{12}u_2(t) \\
&\quad+u_2(t)^{\top}\tilde{c}R_{22}u_2(t)\Big)\dt+ \tilde{c}G(\alpha_T)X(T)^2 \bigg], 
\end{aligned}
\end{equation*}
where $\tilde{c}>\frac{2\bar{c}_2} {\epsilon}$. One can check the latter satisfies Assumptions \ref{ass1}-\ref{ass3}. \end{remark}

\section{Solution to the LQ game \eqref{201}-\eqref{202}}\label{Solution}

In this section, we construct optimal feedback control-strategy pairs through the solutions of a new kind of multidimensional indefinite SRE and a multidimensional linear BSDE.
We first study the solvability of these two BSDEs, and then solve the problem via completing the square method.

To give more concise expressions, we introduce the following notations:
\begin{gather*}
u(\cdot)\triangleq\left[ \begin{array}{cc}
u_1(\cdot)\\
u_2 (\cdot) \\
\end{array} \right]_{m\times 1},~~~
B(t,i)\triangleq\left[ \begin{array}{cc}
B_1(t,i)\\
B_2(t,i) \\
\end{array} \right]_{m\times 1},\\
D(t,i)\triangleq [D_1(t,i),D_2(t,i)]_{n\times m},~~~
R(t,i)\triangleq \left[ \begin{array}{cc}
R_{11}(t,i) &R_{12}(t,i)\\
R_{12}(t,i)^{\top} &R_{22}(t,i) \\
\end{array} \right]_{m\times m},
\end{gather*}
where $m \triangleq m_1+m_2$.
Then, we can rewrite \eqref{201} and \eqref{202} as
\begin{equation*}
\left\{
\begin{aligned}
\dd X(t)=\,&\big[ A(t,\alpha_t)X(t)+B(t,\alpha_t)^{\top}u(t)+b(t,\alpha_t) \big]\dt\\
&+ \big[ C(t,\alpha_t)X(t)+D(t,\alpha_t)u(t) +\sigma(t,\alpha_t) \big]^{\top}\dw(t),\\
X(0)=\,&x\in\mathbb{R},\quad \alpha_0=i_0\in \mathcal{M},
\end{aligned}
\right.
\end{equation*}
and
\begin{equation*}
J_{x,i_0}\big( u_1,u_2)
=\mathbb{E}\bigg[ \int_0^T \Big( K(t,\alpha_t)X(t)^2+u(t)^{\top}R(t,\alpha_t)u(t) \Big)\dt
+ G(\alpha_T)X(T)^2 \bigg].
\end{equation*}

In the classical inhomogeneous LQ control theory, an optimal feedback control is associated with a Riccati equation and a linear equation (see a systematic account in \cite[Chapter 6]{Yong Zhou}). For our LQ game with non-Markovian regime switching, the Riccati equation is a multidimensional BSDE with an indefinite generator and we call it an indefinite SRE.
In addition, the linear equation is a multidimensional BSDE with unbounded coefficients.

\subsection{Solvability of two BSDEs}
To introduce the indefinite SRE and the linear BSDE, we first introduce some notations.
For $(t,i,P, \varphi,\Lambda, \Delta )\in [0,T]\times \mathcal{M}\times\mathbb{R} \times \mathbb{R} \times \mathbb{R}^{n} \times\mathbb{R}^{n}$, we set
\begin{align*}
\widehat{R}(t,i,P) &\triangleq R(t,i)+PD(t,i)^{\top}D(t,i)
\triangleq \begin{bmatrix}
\widehat{R}_{11}(t,i,P) & \widehat{R}_{12}(t,i,P) \\
\widehat{R}_{12}(t,i,P)^{\top} & \widehat{R}_{22}(t,i,P) \\
\end{bmatrix}_{m\times m}\\
&= \begin{bmatrix}
R_{11}+PD_1^{\top}D_1 & R_{12}+PD_1^{\top}D_2 \\
R_{12}^{\top}+PD_2^{\top}D_1 & R_{22}+PD_2^{\top}D_2 \\
\end{bmatrix},\\[10pt]
\widehat{C}(t,i,P,\Lambda)&\triangleq P B(t,i)+D(t,i)^{\top}[PC(t,i)+\Lambda]\\
&\triangleq \begin{bmatrix}
\widehat{C}_{1}(t,i,P,\Lambda) \\
\widehat{C}_{2}(t,i,P,\Lambda) \\
\end{bmatrix} _{m\times 1}
= \begin{bmatrix}
P (B_1+D_1^{\top}C)+D_1^{\top}\Lambda \\
P (B_2+D_2^{\top}C)+D_2^{\top}\Lambda \\
\end{bmatrix},\\[10pt]
\end{align*}
\begin{align*}
\hat{\sigma}(t,i,P,\varphi,\Delta)&\triangleq \varphi B(t,i)+ D(t,i)^{\top}[P\sigma(t,i)+\Delta]\\
&\triangleq \begin{bmatrix}
\hat{\sigma}_{1}(t,i,P,\varphi,\Delta) \\
\hat{\sigma}_{2}(t,i,P,\varphi,\Delta) \\
\end{bmatrix} _{m\times 1}
= \begin{bmatrix}
\varphi B_1+P D_1^{\top}\sigma+D_1^{\top}\Delta \\
\varphi B_2+P D_2^{\top}\sigma+D_2^{\top}\Delta\\
\end{bmatrix},
\end{align*}
where we omit the arguments $(t,i)$ in the coefficients in the last equations of the above definitions.
Henceforth, we often drop the arguments for $\widehat{R}$, $\widehat{C}$, and $\hat{\sigma}$.

\begin{remark}\label{rem:invertible}
The solution to our SRE \eqref{304} below is expected to satisfy $P(\cdot,i)\in \interval$ for all $i\in \mathcal{M}$. This means the SRE \eqref{304} is an indefinite BSDE. In order to solve it, we first establish some priori estimates on the coefficients.
Suppose $P\in \interval$.
Thanks to Assumption \ref{ass2}, for all $i\in \mathcal{M}$, we have
$$\widehat{R}_{11}\leq -\epsilon I_{m_1},\quad \widehat{R}_{22}\geq \epsilon I_{m_2}, \quad
-\frac{1}{\epsilon} I_{m_1}\leq\widehat{R}_{11}^{-1}< 0 I_{m_1},\quad
0 I_{m_2}< \widehat{R}_{22}^{-1}\leq \frac{1}{\epsilon} I_{m_2}.$$
Furthermore, $\widehat{R}$ is an invertible indefinite matrix and
\begin{equation*}
\begin{aligned}
\widehat{R}^{-1}& = \begin{bmatrix}
\widetilde{R}_{11}^{-1} & -\widetilde{R}_{11}^{-1}\widehat{R}_{12}\widehat{R}_{22}^{-1} \\
-\widehat{R}_{22}^{-1}\widehat{R}_{12}^{\top}\widetilde{R}_{11}^{-1} & \widehat{R}_{22}^{-1} +\widehat{R}_{22}^{-1}\widehat{R}_{12}^{\top}\widetilde{R}_{11}^{-1} \widehat{R}_{12}\widehat{R}_{22}^{-1} \\
\end{bmatrix}\\
&= \begin{bmatrix}
\widehat{R}_{11}^{-1}+\widehat{R}_{11}^{-1}\widehat{R}_{12}\widetilde{R}_{22}^{-1}\widehat{R}_{12}^{\top}\widehat{R}_{11}^{-1} & -\widehat{R}_{11}^{-1}\widehat{R}_{12}\widetilde{R}_{22}^{-1} \\
-\widetilde{R}_{22}^{-1}\widehat{R}_{12}^{\top}\widehat{R}_{11}^{-1} & \widetilde{R}_{22}^{-1} \\
\end{bmatrix},
\end{aligned}
\end{equation*}
where $$\widetilde{R}_{11}\triangleq\widehat{R}_{11}-\widehat{R}_{12}\widehat{R}_{22}^{-1}\widehat{R}_{12}^{\top},~~
\widetilde{R}_{22} \triangleq \widehat{R}_{22}-\widehat{R}_{12}^{\top}\widehat{R}_{11}^{-1}\widehat{R}_{12},$$
which satisfy the estimates
$$\widetilde{R}_{11}\leq - \epsilon I_{m_1},\quad\widetilde{R}_{22}\geq \epsilon I_{m_2},\quad
-\frac{1}{\epsilon}I_{m_1}\leq \widetilde{R}_{11}^{-1} < 0I_{m_1},\quad
0I_{m_2}< \widetilde{R}_{22}^{-1} \leq \frac{1}{\epsilon}I_{m_2}. $$
\end{remark}

For $(t,i,P,\varphi,\Lambda,\Delta)\in [0,T]\times \mathcal{M}\times \interval\times \mathbb{R}\times \mathbb{R}^n\times \mathbb{R}^n$, we define
\begin{align*}
H_1(t,i,P,\Lambda) &\triangleq -\widehat{C}^{\top}\widehat{R}^{-1}\widehat{C}=-\widehat{C}_2^{\top}\widehat{R}_{22}^{-1}\widehat{C}_2
-\big[ \widehat{C}_{1}-\widehat{R}_{12}\widehat{R}_{22}^{-1}\widehat{C}_{2} \big]^{\top}\widetilde{R}_{11}^{-1}
\big[ \widehat{C}_{1}-\widehat{R}_{12}\widehat{R}_{22}^{-1}\widehat{C}_{2} \big]\\
&=-\widehat{C}_1^{\top}\widehat{R}_{11}^{-1}\widehat{C}_1
-\big[ \widehat{C}_{2}-\widehat{R}_{12}^{\top}\widehat{R}_{11}^{-1}\widehat{C}_{1} \big]^{\top}\widetilde{R}_{22}^{-1}
\big[ \widehat{C}_{2}-\widehat{R}_{12}^{\top}\widehat{R}_{11}^{-1}\widehat{C}_{1} \big],\\
H_2(t,i,P,\Lambda,\varphi,\Delta) &\triangleq -\widehat{C}^{\top}\widehat{R}^{-1}\hat{\sigma}, \\
H_3(t,i,P,\varphi,\Delta) &\triangleq -\hat{\sigma}^{\top}\widehat{R}^{-1}\hat{\sigma}.
\end{align*}

\begin{remark}\label{rem:growth}
For all $(i,P,\Lambda)\in \mathcal{M} \times\interval\times \mathbb{R}^n$, we have the estimates
$$-\frac{1}{\epsilon}|\widehat{C}_2|^2\leq
-\widehat{C}_2^{\top}\widehat{R}_{22}^{-1}\widehat{C}_2\leq H_1
\leq-\widehat{C}_1^{\top}\widehat{R}_{11}^{-1}\widehat{C}_1
\leq \frac{1}{\epsilon}|\widehat{C}_1|^2.$$
Thus, we obtain from \eqref{bounds on coefficients} the estimate
$$|H_1| \leq\frac{2(c_3\bar{\epsilon}^2+\bar{c}_2|\Lambda|^2)}{\epsilon}.$$
\end{remark}

The indefinite SRE and the linear BSDE for the LQ game \eqref{201}-\eqref{202} are given by
\begin{equation}
\label{304}
\left\{
\begin{aligned}
\dd P(t,i)=\,&-\Big[K(t,i)+P(t,i)\big[2A(t,i)+C(t,i)^{\top}C(t,i)\big]+2C(t,i)^{\top}\Lambda(t,i)\\
&\quad+H_1(t,i,P(t,i),\Lambda(t,i))+\sum_{j\in \mathcal{M}} q_{ij}P(t,j)\Big]\dt +\Lambda(t,i)^{\top}\dw(t),\\
P(T,i)=&\,G(i), \,
P(\cdot,i)\in \interval, \, \text{ for all } \, i\in \mathcal{M},
\end{aligned}
\right.
\end{equation}
and
\begin{equation}
\label{305}
\left\{
\begin{aligned}
\dd \varphi(t,i)=\,&-\Big[P(t,i)\big[b(t,i)+C(t,i)^{\top}\sigma(t,i)\big]+\sigma(t,i)^{\top}\Lambda(t,i)
+A(t,i)\varphi(t,i)\\
&\quad +C(t,i)^{\top}\Delta(t,i)+H_2(t,i,P(t,i),\Lambda(t,i),\varphi(t,i),\Delta(t,i)) \\
&\quad+\sum_{j\in \mathcal{M}} q_{ij}\varphi(t,j)\Big]\dt +\Delta(t,i)^{\top}\dw(t),\\
\varphi(T,i)=\,&0, \,\text{ for all } \, i\in \mathcal{M}.
\end{aligned}
\right.
\end{equation}
The second BSDE depends on the first one, but not vice versa, so they are partially coupled.

\begin{definition} \label{def: solution of SRE}
A vector process $(\bm{P}(\cdot),\bm{\Lambda}(\cdot))=\big(P(\cdot,i),\Lambda(\cdot,i)\big)_{i\in \mathcal{M}}$ is called a solution of multidimensional BSDE \eqref{304}, if it satisfies \eqref{304}, and $\big(P(\cdot,i),\Lambda(\cdot,i)\big)\in L_{\mathcal{F}^W}^{\infty}(0,T;\mathbb{R})\times L_{\mathcal{F}^W}^{2}(0,T;\mathbb{R}^{n})$
for all $i \in \mathcal{M}$. The solution of \eqref{305} is defined similarly.
\end{definition}

We now study the solvability of \eqref{304} and \eqref{305}. Indeed, the solvability of \eqref{305} is already known in the literature. The main difficulty of our paper is to establish the solvability of \eqref{304}.
It is a highly nonlinear multidimensional BSDE and the invertible matrix $\widehat{R}$ in quadratic generator $H_1$ is indefinite, so it is an indefinite SRE.
There are several results on the solvability of indefinite SREs or indefinite quadratic BSDEs (see, e.g., \cite{Hu Zhou 2}, \cite{Qian Zhou}, \cite{Du}).
But up to our knowledge, no existing results could be directly applied to \eqref{304}.
We follow the method in \cite{Xu AAP} to establish the existence result for \eqref{304}.
But this is not straightforward and requires more delicate analysis since $\widehat{R}$ is not definite.

As for the uniqueness, the direct approach using log transformation in \cite{Xu AAP} fails for our indefinite SRE.
In fact, because $\widehat{R}$ is an indefinite matrix, we cannot find a transformation such that the quadratic term in generator is monotone. This is because the goal of the two players in our problem take opposite directions, so that monotonicity is in general losing.
Instead, we establish the uniqueness result for \eqref{304} as a byproduct of Theorem \ref{theorem:1's verification theorem}.
It is a challenge problem to establish the uniqueness by pure BSDE methods, which leaves for our future research.

We now prove the existence of the solution to \eqref{304} by the approximation method in \cite[Lemma 9.6.6]{Zhang book}. Different from \cite{Xu AAP}, the matrix
$\widehat{R}$ in our problem is neither positive- nor negative-semidefinite, so we need two Lipschitz functions to approximate the generator from above and below respectively.
This method is often used to study the solvability of BSDEs in the literature, e.g., \cite{Fan Hu Tang}.

\begin{theorem}[Existence of \eqref{304}] \label{theorem:solvability of ESRE}
The indefinite SRE \eqref{304} admits a solution.
\end{theorem}

\begin{proof}
For $ \bm{P}=(P_1,P_2,\cdots,P_l)^{\top}\in \mathbb{R}^{l}$,
$ \bm{\Lambda}=(\Lambda_1,\Lambda_2,\cdots,\Lambda_l)\in \mathbb{R}^{n\times l}$, $(t,i)\in[0,T]\times \mathcal{M}$, we set
$$g(t,i,\bm{P},\Lambda_i)\triangleq K(t,i)+P_i\big[2A(t,i)+C(t,i)^{\top}C(t,i)\big]+2C(t,i)^{\top}\Lambda_i+\sum_{j\in \mathcal{M}} q_{ij}P_j.$$
Thanks to Assumption \ref{ass2} and Remark \ref{rem:growth}, we can define functions $\overline{H}_1$
and $\underline{H}_1$ such that they are both smooth w.r.t. $P$ and satisfy
\begin{equation*}
\begin{aligned}
\overline{H}_1(t,i,P,\Lambda)&=\underline{H}_1(t,i,P,\Lambda)=0,\,\, \text{ for } |P|\geq 2\bar{\epsilon},\\
\overline{H}_1(t,i,P,\Lambda)&=
-\big[ \widehat{C}_{1}-\widehat{R}_{12}\widehat{R}_{22}^{-1}\widehat{C}_{2} \big]^{\top}\widetilde{R}_{11}^{-1}
\big[ \widehat{C}_{1}-\widehat{R}_{12}\widehat{R}_{22}^{-1}\widehat{C}_{2} \big], \,\, \text{ for } |P|\leq \bar{\epsilon},\\
\underline{H}_1(t,i,P,\Lambda)&=
-\widehat{C}_2^{\top}\widehat{R}_{22}^{-1}\widehat{C}_2, \,\, \text{ for } |P|\leq \bar{\epsilon},
\end{aligned}
\end{equation*}
and
$$0\leq \overline{H}_1(t, i,P,\Lambda)\leq \frac{2(c_3\bar{\epsilon}^2+\bar{c}_2|\Lambda|^2)}{\epsilon}, \quad -\frac{2(c_3\bar{\epsilon}^2+\bar{c}_2|\Lambda|^2)}{\epsilon}\leq
\underline{H}_1(t, i,P,\Lambda) \leq 0.$$
For $k\geq1$, $(t, i,P,\Lambda)\in [0,T]\times \mathcal{M}\times \mathbb{R}\times \mathbb{R}^n$, we define
\begin{equation*}
\begin{aligned}
\overline{H}_1^k(t,i,P,\Lambda)&\triangleq \inf \limits_{(\widetilde{P}, \widetilde{\Lambda})\in \mathbb{R}\times \mathbb{R}^n}
\Big\{ \overline{H}_1(t,i,\widetilde{P},\widetilde{\Lambda})+k|P-\widetilde{P}|+k|\Lambda-\widetilde{\Lambda}| \Big\},\\
\underline{H}_1^k(t,i,P,\Lambda)&\triangleq \sup \limits_{(\widetilde{P}, \widetilde{\Lambda})\in \mathbb{R}\times \mathbb{R}^n}
\Big\{ \underline{H}_1(t,i,\widetilde{P},\widetilde{\Lambda})-k|P-\widetilde{P}|-k|\Lambda-\widetilde{\Lambda}| \Big\}.
\end{aligned}
\end{equation*}
Note that $\overline{H}_1^k$ is increasing to $\overline{H}_1$ and $\underline{H}_1^k$ is decreasing to $\underline{H}_1$ as $k\to\infty$.
For all $k\geq1$ and $(t, i,P,\Lambda)\in [0,T]\times \mathcal{M}\times \mathbb{R}\times \mathbb{R}^n$, we have
\begin{equation*}
0\leq \overline{H}_1^k(t,i,P,\Lambda)
\leq \frac{2(c_3\bar{\epsilon}^2+\bar{c}_2|\Lambda|^2)}{\epsilon},\quad
-\frac{2(c_3\bar{\epsilon}^2+\bar{c}_2|\Lambda|^2)}{\epsilon}
\leq \underline{H}_1^k(t,i,P,\Lambda)\leq 0.
\end{equation*}

Clearly, for any $k$, $\bar{k}\geq1$, $g$, $\overline{H}_1^{\bar{k}}$ and $\underline{H}_1^k$ are uniformly Lipschitz in $(\bm{P},\bm{\Lambda})$, so there exists a unique solution
$(\bm{P}^{k,\bar{k}}(\cdot),\bm{\Lambda}^{k,\bar{k}}(\cdot))=\big(P^{k,\bar{k}}(\cdot,i), \Lambda^{k,\bar{k}}(\cdot,i)\big)_{i\in \mathcal{M}}$
in the space $L^2_{\mathcal{F}}(0,T;\mathbb{R}^{l}) \times L^2_{\mathcal{F}}(0,T;\mathbb{R}^{n\times l}) $
to the following BSDE:
\begin{align*} 
\begin{cases}
\dd P^{k,\bar{k}}(t,i)=-\big[g(t,i,\bm{P}^{k,\bar{k}}(t),\Lambda^{k,\bar{k}}(t,i))
+\overline{H}_1^{\bar{k}} (t,i,P^{k,\bar{k}}(t,i),\Lambda^{k,\bar{k}}(t,i)) \\
\qquad\qquad\qquad\qquad+\underline{H}_1^k(t,i,P^{k,\bar{k}}(t,i),\Lambda^{k,\bar{k}}(t,i)) \big]\dt
+\Lambda^{k,\bar{k}}(t,i)^{\top}\dw(t),\\
~P^{k,\bar{k}}(T,i)=G(i), \text{ for all } i\in \mathcal{M}.
\end{cases}
\end{align*}
Because $\underline{H}_1^k$ is decreasing to $\underline{H}_1$ as $k\to\infty$, by the comparison theorem for multidimensional BSDEs in \cite{Hu Peng}, we obtain
$P^{k,\bar{k}}(\cdot,i)$ is decreasing w.r.t. $k$ for any fixed $i\in \mathcal{M}$ and $\bar{k}\geq1$.
Similarly, it is increasing w.r.t. $\bar{k}$ for any fixed $i\in \mathcal{M}$ and $k\geq1$.

For $(t, i,\bm{P},\Lambda)\in [0,T]\times \mathcal{M}\times \mathbb{R}^l\times \mathbb{R}^n$, we define
\begin{gather*}
\underline{g}(t,i,\bm{P}, \Lambda)=c_1\sum_{j=1}^l P_j+2C(t,i)^{\top}\Lambda
-\overline{K}-\frac{2(c_3\bar{\epsilon}^2+c_2|\Lambda|^2)}{\epsilon}, \\[-4pt]
\overline{g}(t,i,\bm{P}, \Lambda)=c_1\sum_{j=1}^l P_j+2C(t,i)^{\top}\Lambda
+\overline{K}+\frac{2(c_3\bar{\epsilon}^2+c_2|\Lambda|^2)}{\epsilon}, \\[-4pt]
-\underline{P}(t,i)=\overline{P}(t,i)=\overline{G} e^{c_1 l(T-t)}
+ (\overline{K}\epsilon+2c_3\bar{\epsilon}^2) ( e^{c_1 l(T-t)}-1 )/ (c_1 l\epsilon).
\end{gather*}
Clearly, for all $(t,i)\in [0,T]\times\mathcal{M}$,
$$-\bar{\epsilon}= \underline{P}(0,i) \leq \underline{P}(t,i), ~~\overline{P}(t,i) \leq \overline{P}(0,i)= \bar{\epsilon}.$$
Moreover, $\big(\underline{P}(\cdot,i),\bm{0}\big)_{i\in \mathcal{M}}$ and $\big(\overline{P}(\cdot,i),\bm{0}\big)_{i\in \mathcal{M}}$ satisfy BSDEs
\begin{equation*}
\left\{
\begin{aligned}
\dd\underline{P}(t,i)&=-\underline{g}(t,i,\bm{\underline{P}}(t), \underline{\Lambda}(t,i))\dt+\underline{\Lambda}(t,i)^{\top}\dw(t),\\
\underline{P}(T,i)&=-\overline{G}, \text{ for all } i\in \mathcal{M},
\end{aligned}
\right.
\end{equation*}
and
\begin{equation*}
\left\{
\begin{aligned}
\dd\overline{P}(t,i)&=-\overline{g}(t,i,\bm{\overline{P}}(t), \overline{\Lambda}(t,i))\dt+\overline{\Lambda}(t,i)^{\top}\dw(t),\\
\overline{P}(T,i)&=\overline{G}, \text{ for all } i\in \mathcal{M}.
\end{aligned}
\right.
\end{equation*} 
For any $k,\bar{k}\geq 1$ and all $(t,i,\Lambda)\in [0,T]\times\mathcal{M}\times \mathbb{R}^n$, we have
\begin{equation*} \begin{aligned}
\underline{g}(t,i,\underline{\bm{P}}(t), \Lambda)&\leq
g(t,i,\underline{\bm{P}}(t), \Lambda)+ \overline{H}_1^{\bar{k}}(t,i,\underline{P}(t,i), \Lambda)
+\underline{H}_1^{k}(t,i,\underline{P}(t,i), \Lambda),\\
g(t,i,\overline{\bm{P}}(t), \Lambda)&+ \overline{H}_1^{\bar{k}}(t,i,\overline{P}(t,i), \Lambda) +\underline{H}_1^{k}(t,i,\overline{P}(t,i), \Lambda)
\leq \overline{g}(t,i,\overline{\bm{P}}(t), \Lambda).
\end{aligned}\end{equation*}
Then, by the comparison theorem for multidimensional BSDEs in \cite{Xu AAP}, we have
$$-\bar{\epsilon} \leq\underline{P}(\cdot,i) \leq P^{k,\bar{k}}(\cdot,i)\leq \overline{P}(\cdot,i)\leq \bar{\epsilon},\quad
\text{for any }\, k,\bar{k}\geq 1 \text{ and }\, i\in\mathcal{M}. $$
By monotonicity, we can define $P^{\bar{k}}(\cdot,i) = \lim\limits_{k\rightarrow \infty}P^{k,\bar{k}}(\cdot,i)$. Then $P^{\bar{k}}(\cdot,i)\in \interval$. Regarding
$(P^{k,\bar{k}}(\cdot,i), \Lambda^{k,\bar{k}}(\cdot,i))$ as the solution of a scalar-valued quadratic BSDE for each $i\in \mathcal{M}$, by \cite[Lemma 9.6.6]{Zhang book},
for each $\bar{k}\geq1$, there exists a process
$\bm{\Lambda}^{\bar{k}}(\cdot) \in L^2_{\mathcal{F}}(0, T;\mathbb{R}^{n\times l})$ such that $(\bm{P}^{\bar{k}}(\cdot),\bm{\Lambda}^{\bar{k}}(\cdot) )$ is a solution to the following BSDE:
\begin{equation*} 
\left\{
\begin{aligned}
\dd P^{\bar{k}}(t,i)=\,&-\big[g(t,i,\bm{P}^{\bar{k}}(t),\Lambda^{\bar{k}}(t,i))
+\overline{H}_1^{\bar{k}} (t,i,P^{\bar{k}}(t,i),\Lambda^{\bar{k}}(t,i)) \\
&\qquad+\underline{H}_1(t,i,P^{\bar{k}}(t,i),\Lambda^{\bar{k}}(t,i)) \big]\dt
+\Lambda^{\bar{k}}(t,i)^{\top}\dw(t),\\
P^{\bar{k}}(T,i)=\,&G(i), \text{ for all } i\in \mathcal{M}.
\end{aligned}
\right.
\end{equation*}
Recall that $P^{k,\bar{k}}(\cdot,i)$ is increasing w.r.t. $\bar{k}$, so we get $P^{\bar{k}}(\cdot,i)$ is increasing w.r.t. $\bar{k}$.
Hence, we can define $P(\cdot,i) = \lim\limits_{\bar{k}\rightarrow \infty}P^{\bar{k}}(\cdot,i)$.
By \cite[Lemma 9.6.6]{Zhang book} again, there exists a process
$\bm{\Lambda}(\cdot) \in L^2_{\mathcal{F}}(0, T;\mathbb{R}^{n\times l})$ such that
$(\bm{P}(\cdot),\bm{\Lambda}(\cdot))=\big(P(\cdot,i),\Lambda(\cdot,i)\big)_{i\in \mathcal{M}}$ satisfies
\begin{equation*}
\left\{
\begin{aligned}
\dd P(t,i)=\,&-\big[ g(t,i,\bm{P}(t),\Lambda(t,i))
+\overline{H}_1 (t,i,P(t,i),\Lambda(t,i)) \\
&\qquad+\underline{H}_1(t,i,P(t,i),\Lambda(t,i) )\big]\dt +\Lambda(t,i)^{\top}\dw(t),\\
P(T,i)=&\,G(i), \,
P(\cdot,i)\in \interval, \, \text{ for all } \, i\in \mathcal{M}.
\end{aligned}
\right.
\end{equation*}
Notice $P(\cdot,i)\in\interval$, so $\overline{H}_1 +\underline{H}_1={H}_1$ for all $i\in\mathcal{M}$.
We see $(\bm{P}(\cdot),\bm{\Lambda}(\cdot))$ indeed satisfies \eqref{304}.
This established the existence of the solution to \eqref{304}.
\end{proof}

Some coefficients in \eqref{305} depend on $\Lambda$, so they are unbounded.
To prove the solvability of \eqref{305}, we need a more precise estimate on $\Lambda$.
Recall the definition of BMO martingale.
For any process $f\in L^{2}_{\mathcal{F}^W}(0,T;\mathbb{R}^n)$, the process $\int_0^{\cdot} f(s)^{\top}\dw(s)$ is a BMO martingale if there exists a positive constant $c$ such that
$$ \mathbb{E}\bigg[ \int_{\tau}^T |f(s)|^2\ds\;\bigg| \;\mathcal{F}^W_{\tau} \bigg]\leq c$$
hold for all $\mathcal{F}^W_t$-stopping times $\tau \leq T$.
From now on, we use $c$ to represent a positive constant independent of $i$ and $t$, which can be different from line to line.
We set
\begin{equation*} 
\begin{aligned}
L^{2,\mathrm{BMO}}_{\mathcal{F}^W}(0,T;\mathbb{R}^n)
=\,&\bigg\{f \in L^{2}_{\mathcal{F}^W}(0,T;\mathbb{R}^n)\;\bigg|
\int_0^{\cdot} f(s)^{\top}\dw(s) \text{ is a BMO martingale}\bigg\}.
\end{aligned}\end{equation*}
For more details about BMO martingale, interested readers can refer to \cite{Kazamaki}.

\begin{lemma} \label{BMO}
Let $(\bm{P}(\cdot),\bm{\Lambda}(\cdot))=\big(P(\cdot,i),\Lambda(\cdot,i)\big)_{i\in \mathcal{M}}$ be a solution of \eqref{304}. Then we have $\Lambda(\cdot,i)\in L^{2,\mathrm{BMO}}_{\mathcal{F}^W}(0,T;\mathbb{R}^n)$ for all $i\in \mathcal{M}$.
\end{lemma}

\begin{proof}
For any $i\in \mathcal{M}$, applying It\^{o}'s formula to $P(\cdot,i)^2$, we get
\begin{equation*} \begin{aligned}
&\mathbb{E}\bigg[ \int_{\tau}^T |\Lambda|^2\dt\;\bigg| \; \mathcal{F}^{W}_{\tau}\bigg]
-\mathbb{E}\big[G(i)^2\;\big|\;\mathcal{F}^{W}_{\tau}\big]+P(\tau,i)^2\\
=\,&\mathbb{E}\bigg[ \int_{\tau}^T \Big(2P\big[P(2A+C^{\top}C) +K+2C^{\top}\Lambda+H_1(P,\Lambda)
+\sum_{j\in \mathcal{M}} q_{ij}P(t,j)\big]\Big)\dt \;\bigg| \; \mathcal{F}^{W}_{\tau} \bigg]\\
\leq\,& c+\mathbb{E}\bigg[ \int_{\tau}^T \Big(c|\Lambda|+\frac{ 2\bar{c}_2|\Lambda|^2}{\epsilon} \Big)\dt
\;\bigg| \; \mathcal{F}^{W}_{\tau} \bigg],
\end{aligned}\end{equation*}
where $\tau \leq T$ is any $\mathcal{F}^{W}_t$-stopping time. By Assumption \ref{ass3}, there exists a constant $a_1$ such that $1-\frac{2\bar{c}_2}{\epsilon}>a_1>0$. Then
$$\mathbb{E}\bigg[ \int_{\tau}^T |\Lambda|^2\dt \;\bigg| \; \mathcal{F}^{W}_{\tau}\bigg]
\leq c+\mathbb{E}\bigg[ \int_{\tau}^T \bigg(a_1|\Lambda|^2+\frac{ 2\bar{c}_2|\Lambda|^2}{\epsilon} \bigg)\dt\;\bigg| \; \mathcal{F}^{W}_{\tau} \bigg].$$
It leads to the bound
$$\mathbb{E}\bigg[ \int_{\tau}^T |\Lambda(t,i)|^2\dt \;\bigg| \; \mathcal{F}^{W}_{\tau}\bigg]\leq \frac{c}{1-\frac{2\bar{c}_2}{\epsilon}-a_1},$$
so $\Lambda(\cdot,i)\in L^{2,\mathrm{BMO}}_{\mathcal{F}^W}(0,T;\mathbb{R}^n)$.
\end{proof}

\begin{theorem}[Solvability of \eqref{305}] \label{theorem:solvability of LBSDE}
The linear BSDE \eqref{305} with unbounded coefficients has a unique solution $(\bm{\varphi}(\cdot),\bm{\Delta}(\cdot))=\big(\varphi(\cdot,i),\Delta(\cdot,i)\big)_{i\in \mathcal{M}}$ and for all $i\in \mathcal{M}$, $(\varphi(\cdot,i),\Delta(\cdot,i))\in L_{\mathcal{F}^W}^{\infty}(0,T;\mathbb{R})\times L_{\mathcal{F}^W}^{2,\mathrm{BMO}}(0,T;\mathbb{R}^{n})$.
\end{theorem}
\begin{proof}
This is a consequence of \cite[Lemma 3.6]{Xu arXiv1}.
\end{proof}

\subsection{The optimal feedback control-strategy pair}

After solving \eqref{304}-\eqref{305},
we can provide optimal control-strategy pairs to the LQ game \eqref{201}-\eqref{202}.

We set
\begin{equation}\label{301}
\left\{
\begin{aligned}
u_1^{*} (t,i,X(t))&=-\widetilde{R}_{11}^{-1}\big\{\big[ \widehat{C}_{1}- \widehat{R}_{12}\widehat{R}_{22}^{-1}\widehat{C}_{2} \big]X(t)
+\hat{\sigma}_1- \widehat{R}_{12}\widehat{R}_{22}^{-1}\hat{\sigma}_2 \big\},\\
\beta_2^{*}(t,i,u_1(t),X(t))&=- \widehat{R}_{22}^{-1}\big[\widehat{R}_{12}^{\top}u_1(t)+\widehat{C}_{2}X(t)+\hat{\sigma}_2 \big],
\end{aligned}
\right.
\end{equation}
and
\begin{equation}\label{302}
\left\{
\begin{aligned}
u_2^{*} (t,i,X(t))&=-\widetilde{R}_{22}^{-1}\big\{\big[ \widehat{C}_{2}- \widehat{R}_{12}^{\top}\widehat{R}_{11}^{-1}\widehat{C}_{1} \big]X(t)
+\hat{\sigma}_2- \widehat{R}_{12}^{\top}\widehat{R}_{11}^{-1}\hat{\sigma}_1 \big\},\\
\beta_1^{*}(t,i,u_2(t),X(t))&=- \widehat{R}_{11}^{-1}\big[\widehat{R}_{12}u_2(t) +\widehat{C}_{1}X(t)+\hat{\sigma}_1\big],
\end{aligned}
\right.
\end{equation}
where the arguments for $\widetilde{R}_{kk}$, $\widehat{R}_{kk'}$, $\widehat{C}_{k}$, $\hat{\sigma}_{k}$, $k,k'\in\{1,2\}$ are solutions of \eqref{304}-\eqref{305}.

As discussed in \cite{Yu}, it is indeed difficult to prove the adaptability of the optimal control-strategy pairs in the case of random coefficients.
In fact, here we can only obtain $u_1^{*}( \cdot, i, X(\cdot)),\,\beta_1^{*}(\cdot, i,u_2(\cdot),X(\cdot)) \in L_{\mathcal{F}}^{2,\text{loc}}(0,T;\mathbb{R}^{m_1})$
and $u_2^{*}( \cdot, i, X(\cdot))$, $ \beta_2^{*}( \cdot, i, u_1(\cdot), X(\cdot))\in L_{\mathcal{F}}^{2,\text{loc}}(0,T;\mathbb{R}^{m_2})$ for all $i\in \mathcal{M}$.
We now use a common method to deal with optimal control problems with random coefficients, that is, localization method plus some convergence theorems, to obtain their square integrability.

\begin{lemma} \label{lem:admissible}
The feedback control-strategy pair of Player $1$ (resp., Player $2$) defined by \eqref{301} (resp., \eqref{302})
is admissible.
\end{lemma}

\begin{proof}
We only prove that \eqref{301} is admissible, since the proof of \eqref{302} is similar.
The proof is divided into two steps:

\textbf{Goal of Step 1.} For each $u_2(\cdot)\in \mathcal{U}_2$, we prove that SDE \eqref{203} has a unique solution with $\pi=u_1^{*}$,
and $u_1^{*}(\cdot,i,X(\cdot))\in L^2_{\mathcal{F}}(0,T; \mathbb{R}^{m_1})$ for all $i\in \mathcal{M}$;

\textbf{Goal of Step 2.} For each $u_1(\cdot)\in \mathcal{U}_1$, we prove that SDE \eqref{204} has a unique solution with $\Pi=\beta_2^{*}$,
and $\beta_2^{*}(\cdot,i,u_1(\cdot),X(\cdot))\in L^2_{\mathcal{F}}(0,T; \mathbb{R}^{m_2})$ for all $i\in \mathcal{M}$.

\textbf{Step 1.} Let $u_2(\cdot)\in \mathcal{U}_2$.
Then SDE \eqref{203} with the control $\pi=u_1^{*}$ is
\begin{equation}\label{307}
\left\{
\begin{aligned}
\dd X(t)=\,&\big[ AX(t)+B_1^{\top}u_1^{*}(t,\alpha_t,X(t))+B_2^{\top}u_2(t)+b \big]\dt\\
&+ \big[ CX(t)+D_1u_1^{*}(t,\alpha_t,X(t))+D_2u_2(t) +\sigma \big]^{\top}\dw(t),\\
X(0)=\,&x\in\mathbb{R},\quad \alpha_0=i_0\in \mathcal{M}.
\end{aligned}
\right.
\end{equation}
For all $i\in \mathcal{M}$, we have
$|u_1^{*} (t,i,X(t))|\leq c \big[ (1+|\Lambda(t,i)|)|X(t)|+1+|\Delta(t,i)|\big]$ and $\Lambda(\cdot,i),\Delta(\cdot,i)\in L_{\mathcal{F}^W}^{2,\mathrm{BMO}}(0,T;\mathbb{R}^{n})$.
By the basic theorem on pp. 756-757 of \cite{L. I. Gal'Chuk}, the above SDE has a unique strong solution.
As $X(\cdot)$ is continuous for almost all sample paths, it is almost surely bounded on $[0, T]$, which guarantees that
$$\int_0^T |u_1^{*}(t,\alpha_t,X(t))| ^2\dt<\infty.$$

By It\^{o}'s lemma for Markovian chain (see \cite{Hu Liang Tang}), we have
\begin{align*}
\dd P(t,\alpha_t)
=\,&-\big[K+P(t,\alpha_t)(2A+C^{\top}C)+2C^{\top}\Lambda(t,\alpha_t) +H_1(P(t,\alpha_t),\Lambda(t,\alpha_t)) \big]\dt\\
&+\Lambda(t,\alpha_t)^{\top}\dw(t) +\sum_{i,j\in \mathcal{M}} \big[P(t,j)-P(t,i)\big]I_{\{\alpha_{t-}=i\}} \dd\tilde{N}_{ij}(t), \\
\dd \varphi(t,\alpha_t)
=\,&-\big[P(t,\alpha_t)(b+C^{\top}\sigma)+\sigma^{\top}\Lambda(t,\alpha_t)
+A\varphi(t,\alpha_t) \\
&+C^{\top}\Delta(t,\alpha_t)
+H_2(P(t,\alpha_t),\Lambda(t,\alpha_t),\varphi(t,\alpha_t),\Delta(t,\alpha_t)) \big]\dt\\
&+\Delta(t,\alpha_t)^{\top}\dw(t)
+\sum_{i,j\in \mathcal{M}} \big[\varphi(t,j)-\varphi(t,i)\big]I_{\{\alpha_{t-}=i\}} \dd\tilde{N}_{ij}(t),
\end{align*}
where $N_{ij}(t)$, $i,j\in \mathcal{M}$, are independent Poisson processes with intensity $q_{ij}$, and $\tilde{N}_{ij}(t)=N_{ij}(t)- q_{ij}t$ are the corresponding compensated Poisson martingales under the filtration $\mathcal{F}_{t}$.

Applying It\^{o}'s formula to $ P(t,\alpha_t)X(t)^2+2\varphi(t,\alpha_t)X(t)$, where $X(\cdot)$ is the solution of \eqref{307},  
we get
\begin{equation*}
\begin{aligned}
&P(t,\alpha_t)X(t)^2+2\varphi(t,\alpha_t)X(t) -P(0,i_0)x^2-2\varphi(0,i_0)x\\
&+\int_0^t \Big(KX^2+(u_1^{*})^{\top}R_{11}u_1^{*} +2(u_1^{*})^{\top}R_{12}u_2+u_2^{\top}R_{22}u_2 \Big)\ds\\
=\,&\int_0^t \Big(\big[u_2-\beta_2^{*}(u_1^{*}) \big]^{\top} \widehat{R}_{22} \big[u_2-\beta_2^{*}(u_1^{*}) \big]
+P\sigma^{\top}\sigma+2( \varphi b+\sigma^{\top}\Delta ) +H_3(P,\varphi,\Delta) \Big)\ds\\
&+\int_0^t \Big(2(PX+\varphi) ( CX+D_1u_1^{*} +D_2u_2+\sigma )
+ X^2\Lambda+2X\Delta \Big)^{\top}\dw(s)\\
&+ \int_0^t \Big(X^2 \sum_{i,j\in \mathcal{M}} \big[P(s,j)-P(s,i)\big]I_{\{\alpha_{s-}=i\}} + 2X \sum_{i,j\in \mathcal{M}} \big[\varphi(s,j)-\varphi(s,i)\big]I_{\{\alpha_{s-}=i\}}\Big) \dd\tilde{N}_{ij}(s).
\end{aligned}
\end{equation*}
The stochastic integrals in above equation are local martingales, so there exists an increasing sequence of stopping times $\{\tau_k\}$ such that $\tau_k\rightarrow +\infty$ as $k\rightarrow \infty$, and
\begin{equation*}
\begin{aligned}
&\mathbb{E}\big[P(T\wedge\tau_k)X(T\wedge\tau_k)^2+2\varphi(T\wedge\tau_k)X(T\wedge\tau_k) \big]-P(0,i_0)x^2-2\varphi(0,i_0)x\\
&\quad\;+\mathbb{E}\bigg[\int_0^{T\wedge\tau_k} \Big(KX^2+(u_1^{*})^{\top}R_{11}u_1^{*} +2(u_1^{*})^{\top}R_{12}u_2
+u_2^{\top}R_{22}u_2 \Big)\ds\bigg]\\
=\,& \mathbb{E}\bigg[\int_0^{T\wedge\tau_k}
\Big(\big[u_2-\beta_2^{*}(u_1^{*}) \big]^{\top} \widehat{R}_{22} \big[u_2-\beta_2^{*}(u_1^{*}) \big]
+P\sigma^{\top}\sigma+2( \varphi b+\sigma^{\top}\Delta ) +H_3 \Big)\ds\bigg].
\end{aligned}
\end{equation*}
From Assumption \ref{ass2} and Remark \ref{rem:invertible}, we obtain
\begin{equation*}
\begin{aligned}
&(\epsilon+\bar{\epsilon} \bar{c}_2)\mathbb{E}\bigg[ \int_0^{T\wedge\tau_k} |u_1^{*}|^2\ds\bigg]+P(0,i_0)x^2+2\varphi(0,i_0)\\
\leq\,&\mathbb{E}\bigg[ \int_0^{T\wedge\tau_k} \Big(KX^2+u_2^{\top}R_{22}u_2
-P\sigma^{\top}\sigma-2( \varphi b+\sigma^{\top}\Delta )-H_3 \Big)\ds \bigg] \\
&+\frac{\epsilon+\bar{\epsilon} \bar{c}_2}{2}\mathbb{E}\bigg[ \int_0^{T\wedge\tau_k} |u_1^{*}|^2\ds\bigg]
+\frac{2}{\epsilon+\bar{\epsilon} \bar{c}_2}\mathbb{E}\bigg[ \int_0^{T\wedge\tau_k} |R_{12}u_2|^2\ds\bigg]\\
&+\mathbb{E}\big[P(T\wedge\tau_k)X(T\wedge\tau_k)^2+2\varphi(T\wedge\tau_k)X(T\wedge\tau_k) \big].
\end{aligned}
\end{equation*}
Letting $k\rightarrow \infty$ and applying the monotone and dominated convergence theorems,
we obtain $u_1^{*}(\cdot,i,X(\cdot))\in L^2_{\mathcal{F}}(0,T; \mathbb{R}^{m_1})$ for all $i \in \mathcal{M}$.

\textbf{Step 2.} Let $u_1(\cdot)\in \mathcal{U}_1$. Then SDE \eqref{204} with the strategy $\Pi=\beta_2^{*}$ is
\begin{equation}\label{308}
\left\{
\begin{aligned}
\dd X(t)=\,&\big[ AX(t)+B_1^{\top}u_1(t)+B_2^{\top}\beta_2^{*}(t,\alpha_t,u_1(t),X(t))+b \big]\dt\\
&+ \big[ CX(t)+D_1u_1(t)+D_2\beta_2^{*}(t,\alpha_t,u_1(t),X(t)) +\sigma \big]^{\top}\dw(t),\\
X(0)=\,&x\in\mathbb{R},\quad \alpha_0=i_0\in \mathcal{M}.
\end{aligned}
\right.
\end{equation}
For all $i\in \mathcal{M}$, we have
$|\beta_2^{*} (t,\omega,i,u_1(t),X(t))|\leq c \big[(1+|\Lambda(t,i)|)|X(t)|+1+|\Delta(t,i)|+|u_1(t)|\big]$.
By the basic theorem on pp. 756-757 of \cite{L. I. Gal'Chuk}, SDE \eqref{308} has a unique strong solution.
And we have 
$$\int_0^T |\beta_2^{*}(t,\alpha_t,u_1(t),X(t))| ^2\dt<\infty.$$

Similar to Step 1, applying It\^{o}'s formula to $ P(t,\alpha_t)X(t)^2+2\varphi(t,\alpha_t)X(t)$, where $X(\cdot)$ is the solution of \eqref{308},  we get
\begin{equation*}
\begin{aligned}
&\mathbb{E}\big[P(T\wedge\tau_k)X(T\wedge\tau_k)^2+2\varphi(T\wedge\tau_k)X(T\wedge\tau_k) \big]-P(0,i_0)x^2-2\varphi(0,i_0)x\\
&+\mathbb{E}\bigg[ \int_0^{T\wedge\tau_k} \Big(KX^2+u_1^{\top}R_{11}u_1 +2u_1^{\top}R_{12}\beta_2^{*}
+(\beta_2^{*})^{\top}R_{22}\beta_2^{*} \Big)\ds\bigg]\\
=\,&\mathbb{E}\bigg[ \int_0^{T\wedge\tau_k} \Big((u_1-u_1^{*})^{\top} \widetilde{R}_{11} (u_1-u_1^{*})
+P\sigma^{\top}\sigma+2(\varphi b+\sigma^{\top}\Delta)+H_3 \Big)\ds\bigg].
\end{aligned}
\end{equation*}
From Assumption \ref{ass2} and Remark \ref{rem:invertible}, we obtain
\begin{equation*}
\begin{aligned}
&(\epsilon+\bar{\epsilon} \bar{c}_2)\mathbb{E}\bigg[ \int_0^{T\wedge\tau_k} |\beta_2^{*}|^2\ds \bigg]-P(0,i_0)x^2-2\varphi(0,i_0)x\\ 
\leq\,& \mathbb{E}\bigg[ \int_0^{T\wedge\tau_k} \Big(-KX^2-u_1^{\top}R_{11}u_1
+P\sigma^{\top}\sigma+2( \varphi b+\sigma^{\top}\Delta )+H_3 \Big)\ds \bigg]\\
&+\frac{\epsilon+\bar{\epsilon} \bar{c}_2}{2}\mathbb{E}\bigg[ \int_0^{T\wedge\tau_k} |\beta_2^{*}|^2\ds\bigg]
+\frac{2}{\epsilon+\bar{\epsilon} \bar{c}_2}\mathbb{E}\bigg[ \int_0^{T\wedge\tau_k} |R_{12}^{\top}u_1|^2\ds\bigg]\\
& -\mathbb{E}\big[ P(T\wedge\tau_k)X(T\wedge\tau_k)^2+2\varphi(T\wedge\tau_k)X(T\wedge\tau_k) \big] .
\end{aligned}
\end{equation*}
Letting $k\rightarrow \infty$ in above, by the monotone and dominated convergence theorems, we obtain that $\beta_2^{*}(\cdot,i,u_1(\cdot),X(\cdot))\in L^2_{\mathcal{F}}(0,T; \mathbb{R}^{m_2})$ for all $i \in \mathcal{M}$.
\end{proof}

\begin{lemma} \label{lem:optimality}
Define the admissible feedback control-strategy pair $(u_1^{*},\beta_2^{*})$ of Player 1 by \eqref{301}. Then we have the following:
\begin{enumerate}[(a)]
\item $J_{x,i_0}( u_1,\beta_2^{*}(u_1)) \leq J_{x,i_0}( u_1,\beta_2(u_1))$ for any $u_1 \in \mathcal{U}_1$ and $\beta_2 \in \mathcal{A}_2$. Moreover, the equation holds if and only if $\beta_2(u_1) = \beta_2^{*}(u_1)$;
\item $J_{x,i_0}( u_1^{*},\beta_2^{*}(u_1^{*})) \geq J_{x,i_0}( u_1,\beta_2^{*}(u_1))$ for any $u_1 \in \mathcal{U}_1$. Moreover, the equation holds if and only if $u_1 = u_1^{*}$.
\end{enumerate}
\end{lemma}

\begin{proof}
Similar to Step 1 in the proof of Lemma \ref{lem:admissible}, applying It\^{o}'s formula to $ P(t,\alpha_t)X(t)^2+2\varphi(t,\alpha_t)X(t)$,
where $X(\cdot)$ is the solution of \eqref{201}, we get
\begin{equation*}
\begin{aligned}
&\mathbb{E}\bigg[\int_0^{T\wedge \tau_k} \Big( KX^2+u^{\top}Ru \Big)\ds+
\ P(T\wedge\tau_k,\alpha_{T\wedge\tau_k})X(T\wedge\tau_k)^2+2\varphi(T\wedge\tau_k,\alpha_{T\wedge\tau_k})X(T\wedge\tau_k) \bigg]\\
=\,&P(0,i_0)x^2+2\varphi(0,i_0)x
+\mathbb{E}\bigg[ \int_0^{T\wedge \tau_k} \Big(KX^2+u_1^{\top}R_{11}u_1+2u_1^{\top}R_{12}u_2+u_2^{\top}R_{22}u_2\\
&+P( CX+D_1u_1+D_2u_2+\sigma )^{\top} ( CX+D_1u_1 +D_2u_2+\sigma)\\
&+ 2(PX+\varphi) (AX+B_1^{\top}u_1 +B_2^{\top}u_2+b)
+2(\Lambda X+\Delta)^{\top} ( CX+D_1u_1 +D_2u_2+\sigma )\\
&-2X\big[ P(b+C^{\top}\sigma)+\sigma^{\top}\Lambda+A\varphi +C^{\top}\Delta
+H_2(P,\Lambda,\varphi,\Delta) \big] \\
&-X^2 \big[K+P(2A +C^{\top}C)+2C^{\top}\Lambda+H_1(P,\Lambda)\big] \Big)\ds\bigg].
\end{aligned}
\end{equation*}
Then, completing the square for $u_2$, we have
\begin{equation*}
\begin{aligned}
&\mathbb{E}\bigg[\int_0^{T\wedge \tau_k} \Big( KX^2+u^{\top}Ru \Big)\ds+
\ P(T\wedge\tau_k,\alpha_{T\wedge\tau_k})X(T\wedge\tau_k)^2+2\varphi(T\wedge\tau_k,\alpha_{T\wedge\tau_k})X(T\wedge\tau_k) \bigg]\\
=\,&P(0,i_0)x^2+2\varphi(0,i_0)x
+\mathbb{E}\bigg[ \int_0^{T\wedge \tau_k} \Big( \big[u_2-\beta_2^{*}(u_1) \big]^{\top} \widehat{R}_{22} \big[u_2-\beta_2^{*}(u_1) \big]
-\beta_2^{*}(u_1)^{\top} \widehat{R}_{22} \beta_2^{*}(u_1)\\
&+u_1^{\top}R_{11}u_1+P( CX+D_1u_1+\sigma)^{\top} (CX+D_1u_1+\sigma)
+ 2(PX+\varphi) (B_1^{\top}u_1 +b) \\
&+2(\Lambda X+\Delta)^{\top} (D_1u_1+\sigma)-X^2 \big[P C^{\top}C+H_1(P,\Lambda)\big]\\
&-2X\big[ P(b+C^{\top}\sigma)+\sigma^{\top}\Lambda +H_2(P,\Lambda,\varphi,\Delta) \big] \Big)\ds\bigg].
\end{aligned}
\end{equation*}
Once again, completing the square for $u_1$, we obtain
\begin{equation*}
\begin{aligned}
&\mathbb{E}\bigg[\int_0^{T\wedge \tau_k} \Big( KX^2+u^{\top}Ru \Big)\ds+
\ P(T\wedge\tau_k,\alpha_{T\wedge\tau_k})X(T\wedge\tau_k)^2+2\varphi(T\wedge\tau_k,\alpha_{T\wedge\tau_k})X(T\wedge\tau_k) \bigg]\\
=\,&P(0,i_0)x^2+2\varphi(0,i_0)x
+\mathbb{E}\bigg[ \int_0^{T\wedge \tau_k} \Big(\big[u_2-\beta_2^{*}(u_1) \big]^{\top} \widehat{R}_{22} \big[u_2-\beta_2^{*}(u_1) \big]\\
&+(u_1-u_1^{*})^{\top} \widetilde{R}_{11} (u_1-u_1^{*})
+P\sigma^{\top}\sigma+2( \varphi b+\sigma^{\top}\Delta )
+H_3(P,\varphi,\Delta) \Big)\ds\bigg].
\end{aligned}
\end{equation*}
Letting $k\rightarrow \infty$, by the dominated convergence theorems, we obtain
\begin{align}\label{3005}
J_{x,i_0}( u_1,u_2)
=\,&P(0,i_0)x^2+2\varphi(0,i_0)x
+\mathbb{E}\bigg[ \int_0^{T} \Big(\big[u_2-\beta_2^{*}(u_1) \big]^{\top} \widehat{R}_{22} \big[u_2-\beta_2^{*}(u_1) \big]\\
&+(u_1-u_1^{*})^{\top} \widetilde{R}_{11} (u_1-u_1^{*}) +P\sigma^{\top}\sigma+2( \varphi b+\sigma^{\top}\Delta )
+H_3(P,\varphi,\Delta) \Big)\ds\bigg].\nn
\end{align}
From Remark \ref{rem:invertible}, we obtain the desired conclusion and finish the proof.
\end{proof}

Next, we give the solution for the LQ game \eqref{201}-\eqref{202}. The proofs of Theorem \ref{theorem:1's verification theorem} and Corollary \ref{corollary: value function} are similar to that of Theorem 2.5 and Corollary 2.7 in \cite{Yu}, and the latter is not repeated here.

\begin{theorem}[Solution for the LQ game \eqref{201}-\eqref{202}] \label{theorem:1's verification theorem}
For any $(x,i_0)\in \mathbb{R} \times \mathcal{M}$, the LQ game \eqref{201}-\eqref{202} admits an optimal control-strategy pair $(u_1^{*}, \beta_2^{*})$ (resp., $(u_2^{*}, \beta_1^{*})$ ) for Player 1's
(resp., Player 2's) value, which is in a
feedback form defined by \eqref{301} (resp., \eqref{302}). Moreover,
the game has a value, given by
\begin{align*}
V(x,i_0)=\;& P(0,i_0)x^2+2\varphi(0,i_0)x
+\int_0^T \mathbb{E}\Big[ P(t,\alpha_t)\sigma(t,\alpha_t)^{\top}\sigma(t,\alpha_t) +2 \varphi(t,\alpha_t)b(t,\alpha_t)\\
&+2\sigma(t,\alpha_t)^{\top}\Delta(t,\alpha_t)+H_3(t, \alpha_t, P(t,\alpha_t),\varphi(t,\alpha_t),\Delta(t,\alpha_t)) \Big]\dt,
\end{align*}
where $\big(P(\cdot,i),\Lambda(\cdot,i)\big)_{i\in \mathcal{M}}$ and $\big(\varphi(\cdot,i),\Delta(\cdot,i)\big)_{i\in \mathcal{M}}$ are solutions of \eqref{304}-\eqref{305}.
\end{theorem}
\begin{proof}
We only prove the theorem for Player 1, and that of Player 2 is similar.
Letting $(u_1^{*}, \beta_2^{*})$ is defined by \eqref{301}. By Lemma \ref{lem:optimality} (ii), we have
\begin{equation}
\label{306}
J_{x,i_0}( u_1^{*},\beta_2^{*}(u_1^{*})) =\sup\limits_{ u_1 \in \mathcal{U}_1 }J_{x,i_0}( u_1,\beta_2^{*}(u_1)).
\end{equation}
By Lemma \ref{lem:optimality} (i), for any $u_1 \in \mathcal{U}_1$, we have
$$J_{x,i_0}( u_1,\beta_2^{*}(u_1)) =\inf\limits_{ \beta_2 \in \mathcal{A}_2 }J_{x,i_0}( u_1,\beta_2(u_1)). $$
Thus
$$J_{x,i_0}( u_1^{*},\beta_2^{*}(u_1^{*})) =\sup\limits_{ u_1 \in \mathcal{U}_1 }\inf\limits_{ \beta_2 \in \mathcal{A}_2 }
J_{x,i_0}( u_1,\beta_2(u_1)). $$
Obviously,
$$J_{x,i_0}( u_1^{*},\beta_2^{*}(u_1^{*}))=\sup\limits_{ u_1 \in \mathcal{U}_1 }\inf\limits_{ \beta_2 \in \mathcal{A}_2 }J_{x,i_0}( u_1,\beta_2(u_1))
\leq \inf\limits_{ \beta_2 \in \mathcal{A}_2 } \sup\limits_{ u_1 \in \mathcal{U}_1 } J_{x,i_0}( u_1,\beta_2(u_1)). $$
On the other hand, from \eqref{306}, we have
$$J_{x,i_0}( u_1^{*},\beta_2^{*}(u_1^{*}))
\geq \inf\limits_{ \beta_2 \in \mathcal{A}_2 } \sup\limits_{ u_1 \in \mathcal{U}_1 } J_{x,i_0}( u_1,\beta_2(u_1)). $$
Therefore, we obtain
$$J_{x,i_0}( u_1^{*},\beta_2^{*}(u_1^{*}))
=\inf\limits_{ \beta_2 \in \mathcal{A}_2 } \sup\limits_{ u_1 \in \mathcal{U}_1 } J_{x,i_0}( u_1,\beta_2(u_1)). $$
By the definition of Player 1's value, $(u_1^{*}, \beta_2^{*})$ is an optimal control-strategy pair of it.
At last,  
$$ V_1(x,i_0)=P(0,i_0)x^2+2\varphi(0,i_0)x
+ \int_0^T \mathbb{E}\Big[ P\sigma^{\top}\sigma+2(\varphi b+\sigma^{\top}\Delta )
+H_3(P,\varphi,\Delta)\Big]\dt $$
comes from the result of completing the square (please refer to \eqref{3005}).
\end{proof}

\begin{theorem} \label{theorem:Uniqueness of BSDE}
The indefinite SRE \eqref{304} admits a unique solution.
\end{theorem}
\begin{proof}
Let $(\bm{P}(\cdot),\bm{\Lambda}(\cdot))$ and $(\widetilde{\bm{P}}(\cdot),\widetilde{\bm{\Lambda}}(\cdot))$ be solutions of \eqref{304}.
We introduce a family of LQ games parameterized by $(s,x,i_0)\in [0,T)\times \mathbb{R}\times \mathcal{M}$, in which the original LQ game \eqref{201}-\eqref{202} is embedded.

Consider the following control system over $[s,T]$:
\begin{equation*}
\left\{
\begin{aligned}
\dd X(t)=\,&\big[ AX(t)+B^{\top}u(t)+b \big]\dt
+ \big[ CX(t)+Du(t) +\sigma \big]^{\top}\dw(t),\\
X(s)=\,&x\in\mathbb{R},\quad \alpha_s=i_0\in \mathcal{M}.
\end{aligned}
\right.
\end{equation*}
The objective functional is the following $\mathcal{F}_s$-measurable random variable:
\begin{equation*}
J_{s,x,i_0}\big( u_1,u_2)
=\mathbb{E}\bigg[ \int_s^T \Big( KX(t)^2+u(t)^{\top}Ru(t) \Big)\dt + G(\alpha_T)X(T)^2\;\bigg|\; X(s)=x, \alpha_s=i_0 \bigg].
\end{equation*}
For $k\in \{1,2\}$, the admissible control set $\mathcal{U}^s_k= L^2_{\mathcal{F}}(s,T;\mathbb{R}^{m_k})$ and the admissible strategy set $\mathcal{A}^s_k$ are defined similar to Definition \ref{def:admissible strategy} with the initial time being $s$.

For $(s,x,i_{0})\in [0,T)\times \mathbb{R}\times \mathcal{M}$,
Player 1's value and Player 2's value are the following $\mathcal{F}_s$-measurable random variables:
\begin{equation*}
\begin{aligned}
V_1(s,x,i_0)&\triangleq \operatorname*{ess\,inf} \limits_{\beta_2 \in \mathcal{A}_2^s}\, \operatorname*{ess\,sup} \limits_{u_1\in \mathcal{U}_1^s} J_{s,x,i_0}\big( u_1,\beta_2( u_1)),\\
V_2(s,x,i_0) &\triangleq \operatorname*{ess\,sup} \limits_{\beta_1\in \mathcal{A}_1^s}\, \operatorname*{ess\,inf} \limits_{u_2 \in \mathcal{U}_2^s} J_{s,x,i_0}\big( \beta_1(u_2), u_2 ).
\end{aligned}
\end{equation*}
Similar to Theorem \ref{theorem:1's verification theorem}, we have
$V_1(s,x,i_0)=V_2(s,x,i_0)$ and both equal to
\begin{align*}
P(s,i_0)x^2+2\varphi(s,i_0)x
+ \mathbb{E}\Big[ \int_s^T \big( &P\sigma^{\top}\sigma
+2\big[ \varphi b +\sigma^{\top}\Delta \big]\\
&\quad+H_3( P,\varphi,\Delta)\big) \dt \;\Big|\; X(s)=x, \alpha_s=i_0 \Big],
\end{align*}
where $\big(P(\cdot,i),\Lambda(\cdot,i)\big)_{i\in \mathcal{M}}$ and
$\big(\varphi(\cdot,i),\Delta(\cdot,i)\big)_{i\in \mathcal{M}}$ are solutions of \eqref{304}-\eqref{305}.
Because Player 1's value is unique, we get $P(s,i)=\widetilde{P}(s,i)$ for all $(s,i)\in[0,T]\times \mathcal{M}$.

On other hand,
define $\overline{P}(\cdot,i)\triangleq P(\cdot,i)-\widetilde{P}(\cdot,i)$, $\overline{\Lambda}(\cdot,i)\triangleq \Lambda(\cdot,i)-\widetilde{\Lambda}(\cdot,i) $ for all $i\in \mathcal{M}$.
Applying It\^{o}'s formula to $\overline{P}(\cdot,i)^2$, 
and using $\overline{P}(t,i)=0$ for all $(t,i)\in[0,T]\times \mathcal{M}$, we obtain
$\mathbb{E}\Big[ \int_{0}^T \big|\overline{\Lambda}(t,i)\big|^2 \dt \Big]=0$, implying
$\Lambda(t,i)=\widetilde{\Lambda}(t,i) $ for all $(t,i)\in[0,T]\times \mathcal{M}$.
The proof is complete.
\end{proof}

\begin{corollary} \label{corollary: value function}
For any $(x,i_0)\in \mathbb{R} \times \mathcal{M}$, we have
$u_1^{*}=\beta_1^{*}(u_2^{*})$, $u_2^{*}=\beta_2^{*}(u_1^{*})$,
where $u_1^{*}$, $u_2^{*}$, $\beta_1^{*}$, $\beta_2^{*}$ are defined by \eqref{301}-\eqref{302}.
Moreover, the value of the LQ game \eqref{201}-\eqref{202} satisfies $V(x,i_0)= J_{x,i_0}\big( u_1^{*}, u_2^{*})$.
\end{corollary}

\section{Constrained zero-sum SLQD game}\label{Game}

In the previous section, we solved the LQ game \eqref{201}-\eqref{202} for inhomogeneous systems without control constraints.
In this section, we study the game for homogeneous systems, giving an advantage that we can introduce closed convex cone control constraints.

Mathematically, we assume $b(t,i)=0$ and $\sigma(t,i)=\bm{0}$ for all $(t,i)\in [0,T]\times\mathcal{M}$ throughout this section. Then \eqref{201} becomes a homogeneous system:
\begin{equation}\label{201b}
\left\{
\begin{aligned}
\dd X(t)=\,&\big[ A(t,\alpha_t)X(t)+B(t,\alpha_t)^{\top}u(t) \big]\dt + \big[ C(t,\alpha_t)X(t)+D(t,\alpha_t)u(t) \big]^{\top}\dw(t),\\
X(0)=\,&x\in\mathbb{R},\quad \alpha_0=i_0\in \mathcal{M}.
\end{aligned}
\right.
\end{equation} 
\begin{definition}\label{closed convex cone}
A set $\Gamma$ is a closed convex cone if:
\begin{enumerate}[(a)]
\item $\Gamma$ is closed;
\item for all $\lambda\geq 0$, if $u\in \Gamma$, then $\lambda u \in \Gamma$;
\item for all $\theta\in [0,1]$, if $u,v\in \Gamma$, then $ \theta u+(1-\theta)v \in \Gamma$.
\end{enumerate}
\end{definition} 
Let $\Gamma_1\in\mathbb{R}^{m_1}$, $\Gamma_2\in\mathbb{R}^{m_2}$ be two closed convex cones.
For the LQ game \eqref{201b}-\eqref{202} with control constraint set $(\Gamma_1, \Gamma_2)$, the admissible control sets are defined as
$$\widetilde{\mathcal{U}}_k= \big\{u_k(\cdot)\in L^2_{\mathcal{F}}(0,T;\mathbb{R}^{m_k})\;\big|\;u_k(\cdot)\in\Gamma_k \big\},
\, k\in\{1,2\}.$$
Definitions of admissible strategy are similar to Definition \ref{def:admissible strategy} with $\mathcal{U}_k$ being replaced by $\widetilde{\mathcal{U}}_k$ and we denote the set of all admissible strategies for Player $k$ by
$\widetilde{\mathcal{A}}_k$, $k \in\{ 1,2\}$.
For the constrained LQ game \eqref{201b}-\eqref{202}, Player 1's value $\widetilde{V}_1(x,i_0)$ and Player 2's value $\widetilde{V}_2(x,i_0)$, the optimal control-strategy pairs, and the value of the game are defined similarly to that of the LQ game \eqref{201}-\eqref{202}. 

As before, we first introduce some notations.
For $(t,i,P,\Lambda,v_1,v_2)\in [0,T]\times \mathcal{M}\times \interval\times \mathbb{R}^n\times \Gamma_1\times\Gamma_2$, $k\in\{1,2\}$, we define
\begin{equation*}
\begin{aligned}
f_{1k}(t,i,P,\Lambda,v_2)&= \max\limits_{v_1\in \Gamma_1}
\big\{ v_1^{\top}\widehat{R}_{11}v_1 +2v_1^{\top}\widehat{R}_{12}v_2 -2(-1)^{k} \widehat{C}_{1}^{\top}v_1\big\},\\
f_{2k}(t,i,P,\Lambda,v_1)&= \min \limits_{ v_2\in \Gamma_2}
\big\{ v_2^{\top}\widehat{R}_{22}v_2+2v_1^{\top}\widehat{R}_{12}v_2 -2(-1)^{k} \widehat{C}_{2}^{\top}v_2 \big\},\\
\widetilde{H}_{1k}(t,i,P,\Lambda) &= \max \limits_{v_1\in \Gamma_1} \big\{ v_1^{\top}\widehat{R}_{11}v_1 -2(-1)^{k} \widehat{C}_{1}^{\top}v_1
+f_{2k}(t,i,P,\Lambda,v_1)\big\},\\
\widetilde{H}_{2k}(t,i,P,\Lambda) &= \min \limits_{ v_2\in \Gamma_2} \big\{ v_2^{\top}\widehat{R}_{22}v_2 -2(-1)^{k} \widehat{C}_{2}^{\top}v_2
+f_{1k}(t,i,P,\Lambda,v_2)\big\}.
\end{aligned}
\end{equation*}
Henceforth, we drop some arguments for $f_{kk'}$, $\widetilde{H}_{kk'}$, $k,k'\in \{1,2\}$.

\begin{remark}\label{rem:tildeH growth}
For all $(i,P,\Lambda,v_1,v_2) \in\mathcal{M}\times \interval\times \mathbb{R}^n\times\Gamma_1\times\Gamma_2$, $k,k'\in \{1,2\}$, since $\bm{0}\in \Gamma_1$, $\Gamma_2$, we have
$ f_{1k}\geq 0$, $f_{2k}\leq 0$. So we get
\begin{align*}
-\frac{2(c_3\bar{\epsilon}^2+\bar{c}_2|\Lambda|^2)}{\epsilon}
&\leq\min \limits_{ v_2\in \Gamma_2} \{ \epsilon|v_2|^2 -2 |\widehat{C}_{2}| |v_2| \}\\
&\leq \widetilde{H}_{kk'}\leq \max \limits_{v_1\in \Gamma_1} \{ -\epsilon|v_1|^2 +2 |\widehat{C}_{1}| |v_1|\}
\leq\frac{2(c_3\bar{\epsilon}^2+\bar{c}_2|\Lambda|^2)}{\epsilon}.
\end{align*}
Thus, for all $(i,P,\Lambda)\in \mathcal{M} \times\interval\times \mathbb{R}^n$, we have
$$|\widetilde{H}_{kk'}| \leq\frac{2(c_3\bar{\epsilon}^2+\bar{c}_2|\Lambda|^2)}{\epsilon}, \, k,k'\in \{1,2\}.$$
\end{remark}

\begin{lemma}[Minimax theorem] \label{Minimax theorem}
For $(t,i,P,\Lambda,v_1,v_2)\in [0,T]\times\mathcal{M}\times \interval\times \mathbb{R}^n\times\Gamma_1\times\Gamma_2$, $k\in \{1,2\}$, we have
\begin{align*}
\widetilde{H}_{1k}=\widetilde{H}_{2k}=\widetilde{H}_{k}(t,i,P,\Lambda)\triangleq\,&\max \limits_{v_1\in \Gamma_1\atop |v_1|\leq c(1+|\Lambda|)} \min \limits_{ v_2\in \Gamma_2 \atop |v_2|\leq c(1+|\Lambda|)}
\mathcal{H}_{k}(t,i,P,\Lambda,v_1,v_2)\\
=\,&\min \limits_{ v_2\in \Gamma_2 \atop |v_2|\leq c(1+|\Lambda|)} \max \limits_{v_1\in \Gamma_1 \atop |v_1|\leq c(1+|\Lambda|)}
\mathcal{H}_{k}(t,i,P,\Lambda,v_1,v_2),
\end{align*}
where $c$ is any sufficiently large constant, and
$$\mathcal{H}_{k}(t,i,P,\Lambda,v_1,v_2)\triangleq v_1^{\top}\widehat{R}_{11}v_1
-2(-1)^{k} \widehat{C}_{1}^{\top}v_1+v_2^{\top}\widehat{R}_{22}v_2+2v_1^{\top}\widehat{R}_{12}v_2 -2(-1)^{k} \widehat{C}_{2}^{\top}v_2.$$
\end{lemma}

\begin{proof}
We only prove
\begin{equation*} 
\widetilde{H}_{11}(t,i,P,\Lambda)
=\max \limits_{v_1\in \Gamma_1\atop |v_1|\leq c(1+|\Lambda|)} \min \limits_{ v_2\in \Gamma_2 \atop |v_2|\leq c(1+|\Lambda|)}
\mathcal{H}_{1}(t,i,P,\Lambda,v_1,v_2).
\end{equation*}
The proofs for other identities are similar.
The switching between $\max$ and $\min$ is trivially due to the compactness and continuity (see Sion's Minimax theorem \cite{Sion}).

On one hand, for all $(i,P,\Lambda,v_1,v_2)\in \mathcal{M}\times \interval\times \mathbb{R}^n \times\Gamma_1\times \Gamma_2$, we have
$$ v_2^{\top}\widehat{R}_{22}v_2+2v_1^{\top}\widehat{R}_{12}v_2 +2 \widehat{C}_{2}^{\top}v_2
\geq \epsilon|v_2|^2-c\epsilon (1+|v_1|+|\Lambda|)|v_2|$$
for any sufficiently large constant $c>0$.
Hence if $|v_2|> c(1+|v_1|+|\Lambda|)$,
then $$\epsilon|v_2|^2-c\epsilon(1+|v_1|+|\Lambda|)|v_2|>0\geq f_{21},$$ which implies that
\begin{equation*}
\begin{aligned}
f_{21}(t,i,P,\Lambda,v_1)&= \min \limits_{ v_2\in \Gamma_2\atop |v_2|\leq c(1+|v_1|+|\Lambda|)}
\big\{ v_2^{\top}\widehat{R}_{22}v_2+2v_1^{\top}\widehat{R}_{12}v_2 +2 \widehat{C}_{2}^{\top}v_2 \big\}.
\end{aligned}
\end{equation*}
On the other hand, since $f_{21}(v_1)\leq 0$, we have
$$v_1^{\top}\widehat{R}_{11}v_1 +2 \widehat{C}_{1}^{\top}v_1+f_{21}(v_1)\leq -\epsilon|v_1|^2+c(1+|\Lambda|)|v_1|.$$
Hence if $$|v_1|>\frac{c(1+|\Lambda|)+\sqrt{c^2(1+|\Lambda|)^2+8(c_3\bar{\epsilon}^2+\bar{c}_2|\Lambda|^2)}}{2\epsilon},$$
then $$-\epsilon|v_1|^2+c(1+|\Lambda|)|v_1|< -\frac{2(c_3\bar{\epsilon}^2+\bar{c}_2|\Lambda|^2)}{\epsilon} \leq \widetilde{H}_{11},$$ which leads to the desired expression for $\widetilde{H}_{11}$.
\end{proof}

Because of the cone constraint, the indefinite SRE for the LQ game \eqref{201b}-\eqref{202}
is not a single BSDE, but consists of a pair of decoupled BSDEs, which are given by
\begin{equation}
\label{404}
\left\{
\begin{aligned}
\dd P_{k}(t,i)=\,&-\Big[K(t,i)+P_{k}(t,i)\big[2A(t,i)+C(t,i)^{\top}C(t,i)\big]+2C(t,i)^{\top}\Lambda_{k}(t,i)\\
&+\widetilde{H}_{k}(t,i,P_{k}(t,i),\Lambda_{k}(t,i))+\sum_{j\in \mathcal{M}} q_{ij}P_{k}(t,j)\Big]\dt +\Lambda_{k}(t,i)^{\top}\dw(t),\\
P_{k}(T,i)=\,&G(i),\, P_{k}(\cdot,i)\in \interval, \, \text{ for all } \, i\in \mathcal{M},\, k\in\{1,2\}.
\end{aligned}
\right.
\end{equation}

The solutions of \eqref{404} are defined similarly to Definition \ref{def: solution of SRE}.
Similar to Theorem \ref{theorem:solvability of ESRE} and Theorem \ref{theorem:Uniqueness of BSDE}, we can get the solvability of \eqref{404}.

\begin{theorem}[Solvability of \eqref{404}]\label{Constrained BSDE}
The indefinite SREs \eqref{404} have unique solutions $\big(P_{k}(\cdot,i),\Lambda_{k}(\cdot,i)\big)_{i\in \mathcal{M}}$, and $(P_{k}(\cdot,i),\Lambda_{k}(\cdot,i))\in L^{\infty}_{\mathcal{F}^W}(0,T;\mathbb{R})\times L^{2,\mathrm{BMO}}_{\mathcal{F}^W}(0,T;\mathbb{R}^n)$ for all $i\in \mathcal{M}$, $k\in\{1,2\}$.
\end{theorem}

For $(t,i,P,\Lambda,v_1,v_2)\in [0,T]\times \mathcal{M}\times \interval\times \mathbb{R}^n\times \Gamma_1\times\Gamma_2$, $k\in\{1,2\}$, we define
\begin{equation*}
\begin{aligned}
\hat{v}_{1k} (t,i,P,\Lambda)&\triangleq\mathop{\arg \max}\limits_{v_1\in \Gamma_1}
\big\{ v_1^{\top}\widehat{R}_{11}v_1 -2(-1)^{k} \widehat{C}_{1}^{\top}v_1+f_{2k}(v_1)\big\},\\
\hat{v}_{2k} (t,i,P,\Lambda)&\triangleq\mathop{\arg \min}\limits_{ v_2\in \Gamma_2}
\big\{ v_2^{\top}\widehat{R}_{22}v_2-2(-1)^{k} \widehat{C}_{2}^{\top}v_2+f_{1k}(v_2)\big\},\\
\hat{\beta}_{1k}(t,i,P,\Lambda,v_2)&\triangleq\mathop{\arg \max}\limits_{v_1\in \Gamma_1}
\big\{ v_1^{\top}\widehat{R}_{11}v_1 +2v_1^{\top}\widehat{R}_{12}v_2 -2(-1)^{k} \widehat{C}_{1}^{\top}v_1\big\},\\
\hat{\beta}_{2k}(t,i,P,\Lambda,v_1)&\triangleq\mathop{\arg \min} \limits_{ v_2\in \Gamma_2 }
\big\{ v_2^{\top}\widehat{R}_{22}v_2+2v_1^{\top}\widehat{R}_{12}v_2 -2(-1)^{k} \widehat{C}_{2}^{\top}v_2 \big\}.
\end{aligned}
\end{equation*}
Then $|\hat{v}_{kk'}|\leq c(1+|\Lambda|)$, $|\hat{\beta}_{1k}|\leq c(1+|v_{2}|+|\Lambda|)$, $|\hat{\beta}_{2k}|\leq c(1+|v_{1}|+|\Lambda|)$,
where $k,k'\in\{1,2\}$ and some arguments for $\hat{v}_{kk'}$, $\hat{\beta}_{kk'}$ are dropped.
For $k\in\{1,2\}$ and $u_k(\cdot)\in \widetilde{\mathcal{U}}_k$, we define $\mathcal{F}_t$-adapted processes $v_k(t)$ as $\frac{u_k(t)}{|X(t)|}$ when $|X(t)|\neq 0$, and 0 otherwise, 
where $X(\cdot)$ is the corresponding admissible state process.
Notice for $k\in\{1,2\}$, $\Gamma_k$ is a cone, so the process $v_k(\cdot)$ is valued in $\Gamma_k$.
Moreover, we set
\begin{equation}\label{406}
\left\{
\begin{aligned}
u_1^{*}(t,i,X(t))=\,&\hat{v}_{11} (P_{1}(t,i),\Lambda_{1}(t,i))X(t)^{+} +\hat{v}_{12} (P_{2}(t,i),\Lambda_{2}(t,i))X(t)^{-},\\
\beta_2^{*}(t,i,u_1(t),X(t))=\,&\hat{\beta}_{21} (P_{1}(t,i),\Lambda_{1}(t,i),v_1(t))X(t)^{+} +\hat{\beta}_{22} (P_{2}(t,i),\Lambda_{2}(t,i),v_1(t))X(t)^{-},
\end{aligned}
\right.
\end{equation}
and
\begin{equation}\label{407}
\left\{
\begin{aligned}
u_2^{*}(t,i,X(t))=\,&\hat{v}_{21} (P_{1}(t,i),\Lambda_{1}(t,i))X(t)^{+} +\hat{v}_{22} (P_{2}(t,i),\Lambda_{2}(t,i))X(t)^{-},\\
\beta_1^{*}(t,i,u_2(t),X(t))=\,&\hat{\beta}_{11} (P_{1}(t,i),\Lambda_{1}(t,i),v_2(t))X(t)^{+} +\hat{\beta}_{12} (P_{2}(t,i),\Lambda_{2}(t,i),v_2(t))X(t)^{-},
\end{aligned}
\right.
\end{equation}
where for $k\in\{1,2\}$, $\big(P_{k}(\cdot,i),\Lambda_{k}(\cdot,i)\big)_{i\in \mathcal{M}}$ are solutions of \eqref{404}.

\begin{lemma} \label{constained admissible}
The feedback control-strategy pair of Player $1$ (resp., Player $2$) defined by \eqref{406} (resp., \eqref{407})
is admissible.
\end{lemma}
\begin{proof}
We only prove \eqref{406} is admissible, and the proof of \eqref{407} is similar.
By definition, for all $i\in \mathcal{M}$, $k\in\{1,2\}$, $\hat{v}_{1k}\in \Gamma_1$, $\hat{\beta}_{2k} \in \Gamma_2$, so we have $u_1^{*}(\cdot)\in \Gamma_1$, $\beta_2^{*}(\cdot)\in \Gamma_2$.
Similar to Lemma \ref{lem:admissible}, the remaining proof consists of two steps.

\textbf{Step 1.} Let $u_2(\cdot)\in \widetilde{\mathcal{U}}_2$. Then SDE \eqref{203} with $\pi=u_1^{*}$ is
\begin{equation}\label{408}
\left\{
\begin{aligned}
\dd X(t)=\,&\Big[ AX(t)+B_1^{\top}
\big[ \hat{v}_{11}(t)X(t)^{+} +\hat{v}_{12}(t)X(t)^{-} \big]+B_2^{\top}u_2(t) \Big]\dt\\
&+ \Big[ CX(t)+D_1\big[ \hat{v}_{11}(t)X(t)^{+} +\hat{v}_{12}(t)X(t)^{-} \big]
+D_2u_2(t) \Big]^{\top}\dw(t),\\
X(0)=\,&x\in\mathbb{R},\quad \alpha_0=i_0\in \mathcal{M},
\end{aligned}
\right.
\end{equation}
where we abbreviate $\hat{v}_{1k}(t,\alpha_t,P_{k}(t,\alpha_t),\Lambda_{k}(t,\alpha_t))$ to $\hat{v}_{1k}(t)$, $k\in\{1,2\}$.
By the definition of $\hat{v}_{1k}$, for all $i\in \mathcal{M}$, $k\in\{1,2\}$, we have
$|\hat{v}_{1k} (P_{k}(t,i),\Lambda_{k}(t,i))| \leq c (1+|\Lambda_{k}(t,i)|)$.
From Theorem \ref{Constrained BSDE}, we know that $\Lambda_{k}(\cdot,i)\in L^{2,\mathrm{BMO}}_{\mathcal{F}^W} (0,T;\mathbb{R}^n)$ for all $i \in \mathcal{M}$, $k\in\{1,2\}$.
By the basic theorem on pp. 756-757 of \cite{L. I. Gal'Chuk}, SDE \eqref{408} has a unique strong solution. Furthermore, we have
$$|u_1^{*}(t,\alpha_t,X(t))| \leq c \big(1+|\Lambda_{1}(t,\alpha_t)|+|\Lambda_{2}(t,\alpha_t)|\big)|X(t)|,$$ which guarantees that $\int_0^T |u_1^{*}(t,\alpha_t,X(t))| ^2\dt<\infty$.

Similar to Step 1 in Lemma \ref{lem:admissible}, applying It\^{o}'s formula to $ P_{1}(t,\alpha_t)[X(t)^{+}]^2+P_{2}(t,\alpha_t)[X(t)^{-}]^2$, where $X(\cdot)$ is the solution of \eqref{408},  
we get
\begin{align} \label{401}
&\mathbb{E}\big[P_{1}(T\wedge\tau_k)[X(T\wedge\tau_k)^{+}]^2
+P_{2}(T\wedge\tau_k)[X(T\wedge\tau_k)^{-}]^2\big]\\
&+\mathbb{E}\bigg[ \int_0^{T\wedge\tau_k} \Big( KX^2+(u_1^{*})^{\top}R_{11}u_1^{*}+2(u_1^{*})^{\top}R_{12}u_2
+u_2^{\top}R_{22}u_2 \Big)\ds\bigg]\nn\\
=\,&\mathbb{E}\bigg[ \int_0^{T\wedge\tau_k} \phi(P_{1},\Lambda_{1},P_{2},\Lambda_{2},X,u_1^{*},u_2)\ds\bigg]+P_{1}(0,i_0)(x^{+})^2+P_{2}(0,i_0)(x^{-})^2,\nn
\end{align}
where 
\begin{equation*}\label{409}
\begin{aligned}
&\phi(t,i,P_1,\Lambda_1,P_2,\Lambda_2,X,u_1,u_2)\\
\triangleq\,&\big[ u_1^{\top}\widehat{R}_{11}(P_1)u_1+2u_1^{\top}\widehat{R}_{12}(P_1)u_2 +u_2^{\top}\widehat{R}_{22}(P_1)u_2\big]I_{\{X>0\}}\\
&+2X^{+}\big[\widehat{C}_{1}(P_1,\Lambda_1)^{\top}u_1 + \widehat{C}_{2}(P_1,\Lambda_1)^{\top}u_2 \big] -(X^{+})^2\widetilde{H}_{1}(P_1,\Lambda_1) \\
&+\big[ u_1^{\top}\widehat{R}_{11}(P_{2})u_1+2u_1^{\top}\widehat{R}_{12}(P_{2})u_2
+u_2^{\top}\widehat{R}_{22}(P_{2})u_2\big]I_{\{X<0\}}\\
&-2 X^{-}\big[\widehat{C}_{1}(P_{2},\Lambda_{2})^{\top}u_1
+ \widehat{C}_{2}(P_{2},\Lambda_{2})^{\top}u_2\big]
-(X^{-})^2\widetilde{H}_{2}(P_{2},\Lambda_{2}).
\end{aligned}
\end{equation*}
If $X(t) = 0$, then $\phi(P_{1},\Lambda_{1},P_{2},\Lambda_{2},X,u_1^{*},u_2)=0$.
If $X(t) > 0$, on recalling the definition of $\widetilde{H}_{1}$, we have
\begin{equation*}
\phi(P_{1},\Lambda_{1},P_{2},\Lambda_{2},X,u_1^{*},u_2)
=(X^{+})^2 \big[\mathcal{H}_{1}(P_1,\Lambda_1,\hat{v}_{11},v_2) -\widetilde{H}_{1}(P_1,\Lambda_1) \big]\geq 0.
\end{equation*}
If $X(t) < 0$, on recalling the definition of $\widetilde{H}_{2}$, we have
\begin{equation*}
\phi(P_{1},\Lambda_{1},P_{2},\Lambda_{2},X,u_1^{*},u_2)
=(X^{-})^2 \big[ \mathcal{H}_{2}(P_2,\Lambda_2,\hat{v}_{12},v_2)-\widetilde{H}_{2}(P_2,\Lambda_2)\big] \geq0.
\end{equation*}
From \eqref{401} and Assumption \ref{ass2}, we obtain
\begin{equation*} \begin{aligned}
&(\epsilon+\bar{\epsilon} \bar{c}_2)\mathbb{E}\bigg[ \int_0^{T\wedge\tau_k} |u_1^{*}|^2\ds\bigg]
+P_{1}(0,i_0)(x^{+})^2+P_{2}(0,i_0)(x^{-})^2\\
\leq\,&\mathbb{E}\bigg[ \int_0^{T\wedge\tau_k} \Big( KX^2+u_2^{\top}R_{22}u_2 \Big)\ds \bigg]
+\frac{\epsilon+\bar{\epsilon} \bar{c}_2}{2}\mathbb{E}\bigg[ \int_0^{T\wedge\tau_k} |u_1^{*}|^2\ds\bigg]\\
&+\frac{2}{\epsilon+\bar{\epsilon} \bar{c}_2}\mathbb{E}\bigg[ \int_0^{T\wedge\tau_k} |R_{12}u_2|^2\ds\bigg]
+\mathbb{E}\big[P_{1}(T\wedge\tau_k)[X(T\wedge\tau_k)^{+}]^2 +P_{2}(T\wedge\tau_k)[X(T\wedge\tau_k)^{-}]^2 \big].
\end{aligned}\end{equation*}
Letting $k\rightarrow \infty$ in above, by the monotone and dominated convergence theorems, we obtain
$u_1^{*}(\cdot,i,X(\cdot))\in L^2_{\mathcal{F}}(0,T; \mathbb{R}^{m_1})$ for all $i\in \mathcal{M}$.

\textbf{Step 2.}
Let $u_1(\cdot)\in \widetilde{\mathcal{U}}_1$. Then SDE \eqref{204} with $\Pi=\beta_2^{*}$ is
\begin{equation}\label{410}
\left\{
\begin{aligned}
\dd X(t)=\,&\Big[ AX(t)+B_1^{\top}u_1(t)
+B_2^{\top}\big[ \hat{\beta}_{21}(t,\alpha_t)X(t)^{+} +\hat{\beta}_{22}(t)X(t)^{-} \big] \Big]\dt\\
&+ \Big[ CX(t)+D_1u_1(t)
+D_2\big[ \hat{\beta}_{21}(t,\alpha_t)X(t)^{+} +\hat{\beta}_{22}(t)X(t)^{-} \big] \Big]^{\top}\dw(t),\\
X(0)=\,&x\in\mathbb{R},\quad \alpha_0=i_0\in \mathcal{M},
\end{aligned}
\right.
\end{equation}
where we abbreviate $\hat{\beta}_{2k}(t,\alpha_t,P_{k}(t,\alpha_t),\Lambda_{k}(t,\alpha_t),v_1(t))$ to $\hat{\beta}_{2k}(t)$, $k\in\{1,2\}$.
By the definition of $\hat{\beta}_{2k}$, for all $i\in \mathcal{M}$, $k\in\{1,2\}$, we have
$|\hat{\beta}_{2k} (,P_{k}(t,i),\Lambda_{k}(t,i),v_1(t))| \leq c (1+|v_1(t)|+|\Lambda_{k}(t,i)|)$.
By the basic theorem on pp. 756-757 of \cite{L. I. Gal'Chuk}, for any $u_1(\cdot)\in \widetilde{\mathcal{U}}_1$, SDE \eqref{410} has a unique strong solution. Furthermore, we have
$|\beta_2^{*}(t,\alpha_t,X(t),u_1(t))|
\leq c (1+|v_1(t)|+|\Lambda_{1}(t,\alpha_t)|+|\Lambda_{2}(t,\alpha_t)|)|X(t)|$
and
$\int_0^T |\beta_2^{*}(t,\alpha_t,X(t),u_1(t))|^2\dt<\infty$.

Similar to Step 1, applying It\^{o}'s formula to $ P_{1}(t,\alpha_t)[X(t)^{+}]^2+P_{}(t,\alpha_t)[X(t)^{-}]^2$, where $X(\cdot)$ is the solution of \eqref{410},  
we obtain
\begin{equation} \label{402}
\begin{aligned}
&\mathbb{E}\big[P_{1}(T\wedge\tau_k)[X(T\wedge\tau_k)^{+}]^2
+P_{2}(T\wedge\tau_k)[X(T\wedge\tau_k)^{-}]^2\big]\\
&+\mathbb{E}\bigg[ \int_0^{T\wedge\tau_k} \Big( KX^2+u_1^{\top}R_{11}u_1+2u_1^{\top}R_{12}\beta_2^{*}
+(\beta_2^{*})^{\top}R_{22}\beta_2^{*} \Big)\ds\bigg]\\
=\,&\mathbb{E}\bigg[ \int_0^{T\wedge\tau_k} \phi(P_{1},\Lambda_{1},P_{2},\Lambda_{2},X,u_1,\beta_2^{*})\ds\bigg]
+P_{1}(0,i_0)(x^{+})^2+P_{2}(0,i_0)(x^{-})^2.
\end{aligned}\end{equation}
If $X(t) > 0$, on recalling the definition of $\widetilde{H}_{1}$, we have
\begin{equation*}
\phi(P_{1},\Lambda_{1},P_{2},\Lambda_{2},X,u_1,\beta_2^{*})
=(X^{+})^2 \big[ \mathcal{H}_{1}(P_1,\Lambda_1,v_1,\hat{\beta}_{21}) -\widetilde{H}_{1}(P_1,\Lambda_1) \big]\leq 0.
\end{equation*}
If $X(t) < 0$, on recalling the definition of $\widetilde{H}_{2}$, we have
\begin{equation*}
\phi(P_{1},\Lambda_{1},P_{2},\Lambda_{2},X,u_1,\beta_2^{*})
=(X^{-})^2 \big[ \mathcal{H}_{2}(P_2,\Lambda_2,v_1,\hat{\beta}_{22}) -\widetilde{H}_{2}(P_2,\Lambda_2)\big]\leq0.
\end{equation*}
From \eqref{402} and Assumption \ref{ass2}, we obtain
\begin{equation*} \begin{aligned}
&(\epsilon+\bar{\epsilon} \bar{c}_2)\mathbb{E}\bigg[ \int_0^{T\wedge\tau_k} |\beta_2^{*}|^2\ds\bigg]
-P_{1}(0,i_0)(x^{+})^2-P_{2}(0,i_0)(x^{-})^2\\
\leq\,&-\mathbb{E}\bigg[ \int_0^{T\wedge\tau_k} \Big(KX^2+u_1^{\top}R_{11}u_1 \Big)\ds \bigg]
+\frac{\epsilon+\bar{\epsilon} \bar{c}_2}{2}\mathbb{E}\bigg[ \int_0^{T\wedge\tau_k} |\beta_2^{*}|^2\ds\bigg]\\
&+\frac{2}{\epsilon+\bar{\epsilon} \bar{c}_2}\mathbb{E}\bigg[ \int_0^{T\wedge\tau_k} |R_{12}^{\top}u_1|^2\ds\bigg]
-\mathbb{E}\big[ P_{1}(T\wedge\tau_k)[X(T\wedge\tau_k)^{+}]^2 +P_{2}(T\wedge\tau_k)[X(T\wedge\tau_k)^{-}]^2\big].
\end{aligned}\end{equation*}
Letting $k\rightarrow \infty$ in above, by the monotone and dominated convergence theorems, we obtain
$\beta_2^{*}(\cdot,i,X(\cdot),u_1(\cdot))\in L^2_{\mathcal{F}}(0,T; \mathbb{R}^{m_2})$ for all $i\in \mathcal{M}$.
\end{proof}

Next, we give the solution for the constrained LQ game \eqref{201b}-\eqref{202}.

\begin{theorem} [Solution for the constrained LQ game \eqref{201b}-\eqref{202}]\label{Constrained Verification Theorem}
For any $(x,i_0)\in \mathbb{R} \times \mathcal{M}$, the constrained LQ game \eqref{201b}-\eqref{202} admits an optimal control-strategy pair $(u_1^{*}, \beta_2^{*})$ (resp., $(u_2^{*}, \beta_1^{*})$) for Player 1's (resp., Player 2's) value, which is in a
feedback form and defined by \eqref{406} (resp., \eqref{407}). Moreover, the game has a value, given by
$$\widetilde{V}(x,i_0)=P_{1}(0,i_0)(x^{+})^2+P_{2}(0,i_0)(x^{-})^2,$$
where $\big(P_k(\cdot,i),\Lambda_k(\cdot,i)\big)_{i\in \mathcal{M}}$, $k\in\{1,2\}$, are solutions of \eqref{404}.
\end{theorem}

\begin{proof}
We only prove the theorem for Player 1, and that of Player 2 is similar.

On one hand, we get from \eqref{401} that
\begin{equation*}
\begin{aligned}
&\mathbb{E}\big[P_{1}(T\wedge\tau_k,\alpha_{T\wedge\tau_k})[X(T\wedge\tau_k)^{+}]^2
+P_{2}(T\wedge\tau_k,\alpha_{T\wedge\tau_k})[X(T\wedge\tau_k)^{-}]^2\big]\\
&+\mathbb{E}\bigg[ \int_0^{T\wedge\tau_k} \Big( KX^2+(u_1^{*})^{\top}R_{11}u_1^{*}+2(u_1^{*})^{\top}R_{12}u_2
+u_2^{\top}R_{22}u_2 \Big)\ds\bigg]\\
\geq\,& P_{1}(0,i_0)(x^{+})^2+P_{2}(0,i_0)(x^{-})^2,
\end{aligned}\end{equation*}
Letting $k\rightarrow \infty$, by the dominated convergence theorem, we obtain
\begin{equation*}
J_{x,i_0}( u_1^{*},\beta_2(u_1^{*}))\geq P_{1}(0,i_0)(x^{+})^2+P_{2}(0,i_0)(x^{-})^2
\end{equation*}
for any $\beta_2\in\widetilde{\mathcal{A}}_2$.
Then, by the definition of $\widetilde{V}_1(x,i_0)$, we have
\begin{equation*} 
\widetilde{V}_1(x,i_0)
\geq\inf \limits_{\beta_2\in\widetilde{\mathcal{A}}_2} J_{x,i_0}( u_1^{*},\beta_2(u_1^{*}))
\geq P_{1}(0,i_0)(x^{+})^2+P_{2}(0,i_0)(x^{-})^2.
\end{equation*}

On the other hand, we get from \eqref{402} that
\begin{equation*}
\begin{aligned}
&\mathbb{E}\big[P_{1}(T\wedge\tau_k,\alpha_{T\wedge\tau_k})[X(T\wedge\tau_k)^{+}]^2
+P_{2}(T\wedge\tau_k,\alpha_{T\wedge\tau_k})[X(T\wedge\tau_k)^{-}]^2\big]\\
&+\mathbb{E}\bigg[ \int_0^{T\wedge\tau_k} \Big( KX^2+u_1^{\top}R_{11}u_1+2u_1^{\top}R_{12}\beta_2^{*}
+(\beta_2^{*})^{\top}R_{22}\beta_2^{*} \Big)\ds\bigg]\\
\leq\,&P_{1}(0,i_0)(x^{+})^2+P_{2}(0,i_0)(x^{-})^2.
\end{aligned}\end{equation*}
Letting $k\rightarrow \infty$, by the dominated convergence theorem, we obtain
\begin{equation*}
J_{x,i_0}( u_1,\beta_2^{*}(u_1))\leq P_{1}(0,i_0)(x^{+})^2+P_{2}(0,i_0)(x^{-})^2
\end{equation*}
for any $u_1\in\widetilde{\mathcal{U}}_1$.
Then, by the definition of $\widetilde{V}_1(x,i_0)$, we have
\begin{equation*} 
\widetilde{V}_1(x,i_0)
\leq \sup\limits_{u_1\in\widetilde{\mathcal{U}}_1} J_{x,i_0}( u_1,\beta_2^{*}(u_1))\leq P_{1}(0,i_0)(x^{+})^2+P_{2}(0,i_0)(x^{-})^2.
\end{equation*}
Combining the above estimates, we get $\widetilde{V}_1(x,i_0)=P_{1}(0,i_0)(x^{+})^2+P_{2}(0,i_0)(x^{-})^2.$
Noticing 
$$\phi(P_{1},\Lambda_{1},P_{2},\Lambda_{2},X,u_1^{*},\beta_2^{*})=0,$$ 
it is not hard to show $(u_1^{*}, \beta_2^{*})$ is optimal. The proof is complete.
\end{proof}

\begin{corollary}
For any $(x,i_0)\in \mathbb{R} \times \mathcal{M}$, we have
$u_1^{*}=\beta_1^{*}(u_2^{*})$ and $u_2^{*}=\beta_2^{*}(u_1^{*}),$
where $u_1^{*}$, $u_2^{*}$, $\beta_1^{*}$, $\beta_2^{*}$ are defined by \eqref{406} and \eqref{407}.
Moreover, the value of the constrained LQ game \eqref{201b}-\eqref{202} satisfies
$$\widetilde{V}(x,i_0)= J_{x,i_0}\big( u_1^{*}, u_2^{*}),$$
where $\big(P_k(\cdot,i),\Lambda_k(\cdot,i)\big)_{i\in \mathcal{M}}$, $k\in\{1,2\}$, are solutions of \eqref{404}.
\end{corollary}

For the constrained LQ game \eqref{201b}-\eqref{202}, we give two special examples.
\begin{description}
\item[Example 1.]
When $\Gamma_1=\mathbb{R}^{m_1}$ and $\Gamma_2=\mathbb{R}^{m_2}$, the constrained LQ game degenerates into the LQ game discussed in Section \ref{Solution}. In this case,
we have $\widetilde{H}_{1}=\widetilde{H}_{2}=H_1$, and \eqref{404} coincides with \eqref{304}.

\item[Example 2.]
When $\widehat{R}_{12}=\bm{0}$, we have
\begin{align*}
\widetilde{H}_{k}&=
\max \limits_{v_1\in \Gamma_1\atop |v_1|\leq c(1+|\Lambda|)} \big\{ v_1^{\top}\widehat{R}_{11}v_1-2(-1)^{k}\widehat{C}_{1}^{\top}v_1\big\} +\min \limits_{ v_2\in \Gamma_2\atop |v_2|\leq c(1+|\Lambda|)} \big\{v_2^{\top}\widehat{R}_{22}v_2-2(-1)^{k} \widehat{C}_{2}^{\top}v_2\big\},
\end{align*}
and the optimal strategy for Player $k$ with $k\in\{1,2\},$
$$\beta_k^{*}(t,i,X(t))=\hat{\beta}_{k1} (P_{1}(t,i),\Lambda_{1}(t,i))X(t)^{+}
+\hat{\beta}_{k2} (P_{2}(t,i),\Lambda_{2}(t,i))X(t)^{-}$$ does not depend on the opponent's control.
\end{description}

\section{Application to portfolio selection problems}\label{Application}

We consider a financial market consisting of a risk-free asset (the money market instrument or bond) whose price is $S_0$ and two risky securities (the stocks) whose prices are $S_1$ and $S_2$. Assume $W_1$ and $W_2$ are independent standard 1-dimensional Brownian motions. We set $W=(W_1,W_2)^{\top}$.
Their prices are driven by
\begin{equation*}
\left\{
\begin{aligned}
\dd S_0(t)&= r(t,\alpha_t)S_0(t)\dt,\\
\dd S_k(t)&= S_k(t)\big[\mu_k(t,\alpha_t)\dt+\sigma_{k1}(t,\alpha_t)\dw_1(t)+\sigma_{k2}(t,\alpha_t)\dw_2(t)\big],\\
S_0(0)&=s_0,\,\,S_k(0)=s_k,\,\, \alpha_0=i_0\in \mathcal{M},\,\, k\in\{1,2\},
\end{aligned}
\right.
\end{equation*}
where for all $i\in \mathcal{M}$, $r(t,i)$ is the interest rate process, $\mu_k(t, i)$ and
$\sigma_{k}(t, i)=(\sigma_{k1}(t, i),\sigma_{k2}(t, i))$ are the appreciation rate process and volatility rate process of the $k$th risky security corresponding to a market regime $\alpha_t = i$.
And for all $k,k'\in\{1,2\}$ and $i\in \mathcal{M}$, we assume
$r(t, i)$, $\mu_k(t, i)$, $\sigma_{kk'} (t, i) \in L^{\infty}_{\mathcal{F}^W} (0, T;\mathbb{R})$.

Now, we define several constants:
$$
\tilde{\mu}\triangleq\operatorname*{ess\,sup}\limits_{i\in \mathcal{M},\,t\in[0, T]}
\{[\mu_1(t, i)-r(t, i)]^2\vee[\mu_2(t, i)-r(t, i)]^2\},
$$
$$
\bar{\sigma}\triangleq\operatorname*{ess\,sup}\limits_{i\in \mathcal{M},\,t\in[0, T]}
\{\sigma_1(t, i)\sigma_1(t, i)^{\top}\vee \sigma_2(t, i)\sigma_2(t, i)^{\top}\},\,
$$
$$
\underline{\sigma}\triangleq\operatorname*{ess\,inf}\limits_{i\in \mathcal{M},\,t\in[0, T]}
\{\sigma_1(t, i)\sigma_1(t, i)^{\top}\wedge \sigma_2(t, i)\sigma_2(t, i)^{\top}\},
$$
$$\tilde{q}\triangleq \max\limits_{1\leq i,j\leq l}q_{ij},\quad
\tilde{r}\triangleq\operatorname*{ess\,sup}\limits_{i\in \mathcal{M},\,t\in[0, T]} r(t, i),\quad
\epsilon_2\triangleq\frac{ e^{(2\tilde{r}+\tilde{q})lT}-1 +(2\tilde{r}+\tilde{q})le^{(2\tilde{r}+\tilde{q})lT} }
{2(2\tilde{r}+\tilde{q}) l},$$
$$
\epsilon_1\triangleq\frac{2\tilde{\mu}(e^{(2\tilde{r}+\tilde{q})lT}-1)[e^{(2\tilde{r}+\tilde{q})lT}-1
+(2\tilde{r}+\tilde{q})le^{(2\tilde{r}+\tilde{q})lT}]}
{(2\tilde{r}+\tilde{q})^2 l^2}.$$

Suppose there are two players who compete with each other. Both players can invest freely in the risk-free asset, but Player 1 may trade only in the first stock, and similarly, Player 2 may trade only in the second stock. For $k\in\{1,2\}$, let $\pi_k(t)$ denote the amount of Player $k$'s wealth invested in the risky stock $S_k$ at time $t$, and the initial wealth $y_k$ is a positive constant. Both players trade using self-financing strategies, then their wealth processes satisfy
\begin{equation*}
\left\{
\begin{aligned}
\dd Y_k(t)&= \Big[ r(t,\alpha_t)Y_k(t)+\big[\mu_k(t,\alpha_t)-r(t,\alpha_t)\big]\pi_k(t) \Big]\dt+\sigma_{k}(t,\alpha_t)\pi_k(t)\dw(t),\\
Y_k(0)&=y_k,\,\alpha_0=i_0\in \mathcal{M},\quad k\in\{1,2\}.
\end{aligned}
\right.
\end{equation*}
Their wealth difference $X(\cdot)\triangleq Y_1(\cdot)-Y_2(\cdot)$ satisfies
\begin{equation}\label{501}
\left\{
\begin{aligned}
\dd X(t)= \,&\Big[r(t,\alpha_t)X(t)+\big[\mu_1(t,\alpha_t)-r(t,\alpha_t)\big]\pi_1(t) -\big[\mu_2(t,\alpha_t)-r(t,\alpha_t)\big]\pi_2(t)\Big]\dt\\
&+\big[ \sigma_1(t,\alpha_t)\pi_1(t) - \sigma_2(t,\alpha_t)\pi_2(t)\big]\dw(t),\\
X(0)=\,&x\triangleq y_1-y_2,\,\alpha_0=i_0\in \mathcal{M}.
\end{aligned}
\right.
\end{equation}

Player 1 hopes to make his own wealth close to the average wealth of the two players at the end of the investment range. But Player 2 hopes to make the difference $Y_1(T)-\frac{Y_1(T)+Y_2(T)}{2}=\frac{X(T)}{2}$ larger.
At the same time, both players want to take as little risk as possible, which is measured by the amount invested in risky securities. The more money invested in risky securities, the more risk the players take. The functional of this zero-sum game is given as
\begin{equation}\label{502}
J_{x,i_0}( \pi_1,\pi_2)
=\mathbb{E}\bigg[ \int_0^T \Big(-R_{1}(t,\alpha_t)\pi_1(t)^2 + R_{2}(t,\alpha_t)\pi_2(t)^2 \Big)\dt
- \frac{1}{4}X(T)^2 \bigg],
\end{equation}
where for all $i\in \mathcal{M}$, $k\in\{1,2\}$, $R_{k}(t,i)>0$ is Player $k$'s risk weight corresponding to a market regime $\alpha_t= i$. And we assume
$R_{1}(t, i)$, $R_{2} (t, i) \in L^{\infty}_{\mathcal{F}^W} (0, T; \mathbb{R}_{>0})$,
for all $i\in \mathcal{M}$.
In this game, Player 1 aims to maximize functional \eqref{502}, whereas Player 2 aims to minimize
it.
We call this problem the LQ game \eqref{501}-\eqref{502}.

We put the following conditions on the coefficients.

\begin{condition}\label{con1}
$\underline{\sigma}>0$,~
$\operatorname*{ess\,inf}\limits_{i\in \mathcal{M},\,t\in[0, T]} \{R_1(t, i)\wedge R_2(t, i)\}
>\epsilon_1+\bar{\sigma}\epsilon_2$,~
$2\bar{\sigma}<\epsilon_1$.
\end{condition}

For the LQ game \eqref{501}-\eqref{502}, Assumptions \ref{ass1}-\ref{ass3} hold if the coefficients satisfy Condition \ref{con1}.
Next, we consider the LQ game \eqref{501}-\eqref{502} with possible no-shorting portfolio constraints.

\subsection{No portfolio constraint}
In this subsection, we assume there are no trading constraints for both players, namely, $\Gamma_1=\Gamma_2 = \mathbb{R}$.
In this case, \eqref{305} admits a unique solution $(\varphi(\cdot,i),\Delta(\cdot,i))= (0,\bm{0})$, $i\in \mathcal{M}$ and SRE \eqref{304} becomes
\begin{equation}
\label{504}
\left\{
\begin{aligned}
\dd P(t,i)=\,&-\bigg[2rP(t,i)-\frac{\Upsilon(P(t,i),\Lambda(t,i))}{\Theta(P(t,i),\Lambda(t,i))}
+\sum_{j\in \mathcal{M}} q_{ij}P(t,j)\bigg]\dt +\Lambda(t,i)^{\top}\dw(t),\\
P(T,i)=\,&-\frac{1}{4},\,
|P(\cdot,i)|\leq \epsilon_2, \, \text{ for all } \, i\in \mathcal{M},
\end{aligned}
\right.
\end{equation}
where for all $(t,i,P,\Lambda)\in [0,T]\times\mathcal{M}\times[-\epsilon_2,\epsilon_2]\times \mathbb{R}^2$,
$$\Phi_1(t,i,P,\Lambda)\triangleq P(\mu_1-r)+\sigma_1\Lambda,\,\,
\Phi_2(t,i,P,\Lambda)\triangleq -P(\mu_2-r)-\sigma_2\Lambda,$$
$$\Psi_{1}(t,i,P)\triangleq P\sigma_1\sigma_1^{\top}-R_1,\,\,
\Psi_{2}(t,i,P)\triangleq P\sigma_2\sigma_2^{\top}+R_2,\,\,
\Psi_{3}(t,i,P)\triangleq -P\sigma_1\sigma_2^{\top},$$
$$
\Theta(t,i,P)\triangleq \Psi_{1}\Psi_{2}-\Psi_{3}^2<0,\,\,
\Upsilon(t,i,P,\Lambda)\triangleq \Psi_{1}\Phi_2^2+\Psi_{2}\Phi_1^2-2\Psi_{3}\Phi_1\Phi_2.
$$
From Theorem \ref{theorem:solvability of ESRE}, BSDE \eqref{504} admits a unique solution $\big(P(\cdot,i),\Lambda(\cdot,i)\big)_{ i \in \mathcal{M}}$.

\begin{theorem}
Suppose Condition \ref{con1} holds and $\Gamma_1=\Gamma_2 = \mathbb{R}$. For any $(x,i_0)\in \mathbb{R} \times \mathcal{M}$, the unconstrained LQ game \eqref{501}-\eqref{502} admits optimal control-strategy pairs $(\pi_1^{*},\beta_2^{*})$ for Player 1 and $(\pi_2^{*},\beta_1^{*})$ for Player 2, which are
\begin{equation*}
\left\{
\begin{aligned}
\pi_1^{*} (t,i,X(t))&=-\Upsilon_1(P(t,i),\Lambda(t,i)) X(t)/ \Theta(P(t,i)),\\
\beta_2^{*}(t,i,\pi_1(t),X(t))&=-
\big[\Psi_{3}(P(t,i)) \pi_1(t) +\Phi_{2}(P(t,i),\Lambda(t,i))X(t)\big]\big/ \Psi_{2}(P(t,i)),
\end{aligned}
\right.
 \end{equation*}
and \begin{equation*}
\left\{
\begin{aligned}
\pi_2^{*} (t,i,X(t))&=- \Upsilon_2(P(t,i),\Lambda(t,i)) X(t)/ \Theta(P(t,i)),\\
\beta_1^{*}(t,i,\pi_2(t),X(t))&=- \big[\Psi_{3}(P(t,i)) \pi_2(t) +\Phi_{1}(P(t,i),\Lambda(t,i))X(t)\big]\big/ \Psi_{1}(P(t,i)),
\end{aligned}
\right.
\end{equation*}
where $\big(P(\cdot,i),\Lambda(\cdot,i)\big)_{i\in \mathcal{M}}$ is the solution of \eqref{504},
and
$$\Upsilon_1(t,i,P,\Lambda)\triangleq \Psi_{2}\Phi_{1}-\Psi_{3}\Phi_{2},~~\Upsilon_2(t,i,P,\Lambda)\triangleq \Psi_{1}\Phi_{2}-\Psi_{3}\Phi_{1}.$$
Moreover, the unconstrained LQ game has a value, given by
$$V(x,i_0)=P(0,i_0)x^2.$$
\end{theorem}

\begin{condition}\label{con4}
$\sigma_1\sigma_2^{\top}=0$ for all $i\in \mathcal{M}$.
\end{condition}

\begin{remark}
If the coefficients satisfy Conditions \ref{con1}-\ref{con4}, then we have
$$\beta_k^{*}(t,i,X(t))= -\frac{ \Phi_{k}(P(t,i),\Lambda(t,i))X(t)} {\Psi_{k}(P(t,i))},\,\,k\in\{1,2\}.$$
In fact, $\sigma_1\sigma_2^{\top}$ is the correlation coefficient between $\ln S_1$ and $\ln S_2$. If there is no correlation between the risky assets, then the optimal strategies are only the feedback of state, and have nothing to do with the opposite player's portfolio.
\end{remark}

\subsection{Exactly one player is subject to no-shorting constraint}

In this subsection, we study the LQ game \eqref{501}-\eqref{502} when exactly one player is subject to no-shorting constraint.
We first introduce the following BSDEs
\begin{equation}\label{505}
\left\{
\begin{aligned}
\dd P_{k}(t,i)=\,&-\big[2rP_k(t,i) +\widetilde{G}_k(P_k(t,i),\Lambda_k(t,i))
+\sum_{j\in \mathcal{M}} q_{ij}P_k(t,j)\big]\dt +\Lambda_k(t,i)^{\top}\dw(t),\\[-3pt]
P_{k}(T,i)=\,&-\frac{1}{4},\, |P_k(\cdot,i)|\leq \epsilon_2, \, \text{ for all } \, i\in \mathcal{M},\, k=\{1,2,3,4,5,6\},
\end{aligned}
\right.
\end{equation}
where for all $(t,i,P,\Lambda)\in [0,T]\times\mathcal{M}\times[-\epsilon_2,\epsilon_2]\times \mathbb{R}^2$,
\begin{align*}
\widetilde{G}_1(t,i,P,\Lambda) &\triangleq \frac{-(\Upsilon_1^{+})^2-\Theta\Phi_2^2}
{\Theta\Psi_2},~~
\widetilde{G}_2(t,i,P,\Lambda)\triangleq \frac{-(\Upsilon_1^{-})^2-\Theta\Phi_2^2}
{\Theta\Psi_2},\\
\widetilde{G}_3(t,i,P,\Lambda) &\triangleq \frac{-(\Upsilon_2^{+})^2-\Theta\Phi_1^2}
{\Theta\Psi_1},~~
\widetilde{G}_4(t,i,P,\Lambda)\triangle \frac{-(\Upsilon_2^{-})^2-\Theta\Phi_1^2}
{\Theta\Psi_1},\\
\widetilde{G}_5(t,i,P,\Lambda) &\triangleq \big[(\Phi_1^{+})^2-2\Phi_1\Phi_1^{+}\big]\big/\Psi_1 +\big[(\Phi_2^{-})^2+2\Phi_2\Phi_2^{-}\big]\big/\Psi_2,\\
\widetilde{G}_6(t,i,P,\Lambda) &\triangleq \big[(\Phi_1^{-})^2+2\Phi_1\Phi_1^{-}\big]\big/ \Psi_1+ \big[(\Phi_2^{+})^2-2\Phi_2\Phi_2^{+}\big]\big/ \Psi_2.
\end{align*}
From Theorem \ref{Constrained BSDE}, BSDEs \eqref{505} admit unique solutions
$\big(P_k(\cdot,i),\Lambda_k(\cdot,i)\big)_{i\in \mathcal{M}}$ and $(P_k(\cdot,i),\Lambda_k(\cdot,i))\in L_{\mathcal{F}^W}^{\infty}(0,T;\mathbb{R})\times L_{\mathcal{F}^W}^{2,\mathrm{BMO}}(0,T;\mathbb{R}^{2})$ for all $i\in \mathcal{M}$, $k\in\{1,2,3,4,5,6\}$.

\textbf{Case I: Only Player 1 is subject to no-shorting constraint. }  
We assume just Player 1's portfolio is subject to no-shorting constraint, i.e., $\Gamma_1=[0,+\infty)$, $\Gamma_2 = \mathbb{R}$.
In this case, SREs \eqref{404} become \eqref{505} with $k\in\{1,2\}$.

\begin{theorem}
Suppose Condition \ref{con1} holds and $\Gamma_1=[0,+\infty)$, $\Gamma_2 = \mathbb{R}$. For any $(x,i_0)\in \mathbb{R} \times \mathcal{M}$, the constrained LQ game \eqref{501}-\eqref{502} admits optimal control-strategy pairs $(\pi_1^{*},\beta_2^{*})$ for Player 1 and $(\pi_2^{*},\beta_1^{*})$ for Player 2, which are
\begin{equation*}
\left\{
\begin{aligned}
\pi_1^{*} (t,i,X(t))
=\,& - \Upsilon_1(P_1,\Lambda_1)^{+} X(t)^{+} / \Theta(P_1)
- \Upsilon_1(P_2,\Lambda_2)^{-} X(t)^{-} / \Theta(P_2),\\[5pt]
\beta_2^{*}(t,i,\pi_1(t),X(t))=\,&- \big[ \Psi_3(P_1)\pi_1(t)I_{\{X(t)>0\}}+\Phi_2(P_1,\Lambda_1) X(t)^{+} \big] \big/ \Psi_2(P_1)\\[3pt]
&- \big[ \Psi_3(P_2)\pi_1(t)I_{\{X(t)<0\}}-\Phi_2(P_2,\Lambda_2) X(t)^{-} \big] \big/ \Psi_2(P_2),
\end{aligned}
\right.
 \end{equation*}
and \begin{equation*}
\left\{
\begin{aligned}
\pi_2^{*} (t,i,X(t))=\,&
\frac{ \big[\Psi_3(P_1)\Upsilon_1(P_1,\Lambda_1)^{+} - \Phi_2(P_1,\Lambda_1)\Theta(P_1)\big] X(t)^{+} }
{\Psi_2(P_1)\Theta(P_1)}\\[3pt]
&+\frac{ \big[\Psi_3(P_2)\Upsilon_1(P_2,\Lambda_2)^{-} + \Phi_2(P_2,\Lambda_2)\Theta(P_2)\big] X(t)^{-} }
{\Psi_2(P_2)\Theta(P_2)},\\[5pt]
\beta_1^{*}(t,i,\pi_2(t),X(t))=\,&
- \big[\Psi_3(P_1)\pi_2(t)I_{\{X(t)>0\}}+\Phi_1(P_1,\Lambda_1) X(t)^{+}\big]^{+} \big/ \Psi_1(P_1)\\[3pt]
&- \big[ \Psi_3(P_2)\pi_2(t)I_{\{X(t)<0\}}-\Phi_1(P_2,\Lambda_2) X(t)^{-} \big]^{+} \big/ \Psi_1(P_2),
\end{aligned}
\right.
\end{equation*}
where $\big(P_k(\cdot,i),\Lambda_k(\cdot,i)\big)_{i\in \mathcal{M}}$, $k\in\{1,2\}$, are solutions of \eqref{505}.
Moreover, the constrained LQ game has a value, given by
$$V(x,i_0)=P_1(0,i_0)(x^{+})^2+P_2(0,i_0)(x^{-})^2.$$
\end{theorem}

\textbf{Case II: Only Player 2 is subject to no-shorting constraint. } 
We assume just Player 2's portfolio is subject to no-shorting constraint, i.e., $\Gamma_1=\mathbb{R}$, $\Gamma_2 =[0,+\infty) $.
In this case, SREs \eqref{404} become \eqref{505} with $k\in\{3,4\}$.

\begin{theorem}
Suppose Condition \ref{con1} holds and $\Gamma_1=\mathbb{R}$, $\Gamma_2 =[0,+\infty) $. For any $(x,i_0)\in \mathbb{R} \times \mathcal{M}$, the constrained LQ game \eqref{501}-\eqref{502} admits optimal control-strategy pairs $(\pi_1^{*},\beta_2^{*})$ for Player 1
and $(\pi_2^{*},\beta_1^{*})$ for Player 2, which are
\begin{equation*}
\left\{
\begin{aligned}
\pi_1^{*} (t,i,X(t))=\;&\frac{ \big[\Psi_3(P_3)\Upsilon_2(P_3,\Lambda_3)^{+} -\Phi_1(P_3,\Lambda_3)\Theta(P_3)\big] X(t)^{+} }{\Psi_1(P_3)\Theta(P_3)}\\[3pt]
&+\frac{ \big[\Psi_3(P_4)\Upsilon_2(P_4,\Lambda_4)^{-} + \Phi_1(P_4,\Lambda_4)\Theta(P_4)\big] X(t)^{-} }{\Psi_1(P_4)\Theta(P_4)}, \\[5pt]
\beta_2^{*}(t,i,\pi_1(t),X(t))=\;&\big[\Psi_3(P_3)\pi_1(t)I_{\{X(t)>0\}}+\Phi_2(P_3,\Lambda_3) X(t)^{+}\big]^{-} \big/\Psi_2(P_3)\\[3pt]
&+ \big[ \Psi_3(P_4)\pi_1(t)I_{\{X(t)<0\}}-\Phi_2(P_4,\Lambda_4) X(t)^{-} \big]^{-} \big/ \Psi_2(P_4),
\end{aligned} 
\right.
 \end{equation*}
and \begin{equation*}
\left\{
\begin{aligned}
\pi_2^{*} (t,i,X(t))=\;& - \Upsilon_2(P_3,\Lambda_3)^{+} X(t)^{+} /\Theta(P_3)
- \Upsilon_2(P_4,\Lambda_4)^{-} X(t)^{-} / \Theta(P_4), \\[5pt]
\beta_1^{*}(t,i,\pi_2(t),X(t))=\;&- \big[ \Psi_3(P_3)\pi_2(t)I_{\{X(t)>0\}}+\Phi_1(P_3,\Lambda_3) X(t)^{+} \big] \big/ \Psi_1(P_3)\\[3pt]
&- \big[ \Psi_3(P_4)\pi_2(t)I_{\{X(t)<0\}}-\Phi_1(P_4,\Lambda_4) X(t)^{-} \big] \big/ \Psi_1(P_4),
\end{aligned}
\right.
\end{equation*}
where $\big(P_k(\cdot,i),\Lambda_k(\cdot,i)\big)_{i\in \mathcal{M}}$, $k\in\{3,4\}$, are solutions of \eqref{505}.
Moreover, the constrained LQ game has a value, given by
$$V(x,i_0)=P_3(0,i_0)(x^{+})^2+P_4(0,i_0)(x^{-})^2.$$
\end{theorem}

\subsection{Both players are subject to no-shorting constraint}\label{cons3}

In this subsection, we assume both players are subject to no-shorting constraint, i.e. $\Gamma_1=\Gamma_2=[0,+\infty)$.
In this case, SREs \eqref{404} become \eqref{505} with $k\in\{5,6\}$.

\begin{theorem} Suppose Conditions \ref{con1}-\ref{con4} hold and $\Gamma_1=\Gamma_2=[0,+\infty)$. For any $(x,i_0)\in \mathbb{R} \times \mathcal{M}$, the constrained LQ game \eqref{501}-\eqref{502} admits optimal control-strategy pairs $(\pi_1^{*},\beta_2^{*})$ for Player 1 and $(\pi_2^{*},\beta_1^{*})$ for Player 2,
which are
\begin{equation*}
\left\{
\begin{aligned}
\pi_1^{*} (t,i,X(t))=\;&\beta_1^{*}(t,i,X(t))
=\frac{ \Phi_1(P_5,\Lambda_5)^{+} X(t)^{+} } {-\Psi_1(P_5)}
+ \frac{ \Phi_1(P_6,\Lambda_6)^{-} X(t)^{-} } {-\Psi_1(P_6)},\\[5pt]
\pi_2^{*} (t,i,X(t))=\;&\beta_2^{*}(t,i,X(t))=\frac{ \Phi_2(P_5,\Lambda_5)^{-} X(t)^{+} } {\Psi_2(P_5)}+ \frac{ \Phi_2(P_6,\Lambda_6)^{+} X(t)^{-} } {\Psi_2(P_6)},
\end{aligned}
\right.
\end{equation*}
where $\big(P_k(\cdot,i),\Lambda_k(\cdot,i)\big)_{i\in \mathcal{M}}$, $k\in\{5,6\}$, are solutions of \eqref{505}.
Moreover, the constrained LQ game has a value, given by
$$V(x,i_0)=P_5(0,i_0)(x^{+})^2+P_6(0,i_0)(x^{-})^2.$$
\end{theorem}

\section{Conclusion}\label{Conclusion}
In this paper, we studied zero-sum SLQD games for systems with regime switching and random coefficients.
We obtained the optimal feedback control-strategy pairs for the two players via some new kind of multidimensional BSDEs.
The solvability of the indefinite SREs is interesting in its own right from the BSDE theory point of view. For homogeneous systems, we put closed convex cone control constraint and obtained the corresponding optimal feedback control-strategy pairs.
At last, we solved several portfolio selection problems with possible no-shorting constraints in a non-Markovian regime switching market.

There are many possible interesting extensions. For instance,
(1)
The optimal feedback control-strategy pairs in this paper depend on the sample path.
What players can usually observe in practice is the state of the system or another observation process.
So one can consider the problem in a partially observable framework, where controls or strategies are adapted to the observed information.
(2) One can consider the problem with multidimensional state process, in which case one has to study the solvability of matrix-valued indefinite SREs.


\begin{thebibliography}{9}
\bibitem {Bismut}
J. M. Bismut, {\it Linear quadratic optimal stochastic control with random coefficients}, SIAM J. Control Optim., \textbf{14} (1976), pp.~419--444.

\bibitem {Buckdahn Li}
R. Buckdahn and J. Li, {\it Stochastic differential games and viscosity solutions of Hamilton-Jacobi-Bellman-Isaacs equations}, SIAM J. Control Optim., \textbf{47} no.1 (2008), pp.~444--475.


\bibitem {Zhang book}
J. Cvitanic and J. F. Zhang, {\it Contract theory in continuous-time models}, Springer Science and Business Media, 2012.

\bibitem {Du}
K. Du, {\it Solvability conditions for indefinite linear quadratic optimal
stochastic control problems and associated stochastic Riccati equations},
SIAM J. Control Optim., \textbf{53} no.6 (2015), pp.~3673--3689.


\bibitem {Elliott}
R. J. Elliott and N. J. Kalton, {\it Mem. Amer. Math. Soc.: No. 126, The existence
of value in differential games}, American Mathematical Society, Providence, Rhode Island, 1972.

\bibitem {Fan Hu Tang}
S. J. Fan, Y. Hu and S. J. Tang, {\it Multi-dimensional backward stochastic differential equations of diagonally quadratic generators: The general result}, J. Differ. Equ., \textbf{368} no.25 (2023), pp.~105--140.

\bibitem {Fleming}
W. H. Fleming and P. E. Souganidis, {\it On the existence of value functions of two-player, zero-sum stochastic differential games}, Indiana Univ. Math. J., \textbf{38} (1989), pp.~293--314.

\bibitem {L. I. Gal'Chuk}
L. I. Gal'Chuk, {\it Existence and uniqueness of a solution for stochastic equations with respect to semimartingales},
Theory Probab. Appl., \textbf{23} no.4 (1979), pp.~751--763.

\bibitem {Hu Liang Tang}
Y. Hu, G. C. Liang and S. J. Tang, {\it Systems of infinite horizon and ergodic BSDE arising in regime switching forward performance processes}, SIAM J. Control optim., \textbf{58} no.4 (2020), pp.~2503--2534.

\bibitem {Xu AAP}
Y. Hu, X. M. Shi and Z. Q. Xu, {\it Constrained stochastic LQ control with regime switching and application to portfolio selection},
Ann. Appl. Probab., \textbf{32} no.1 (2022), pp.~426--460.

\bibitem {Xu arXiv1}
Y. Hu, X. M. Shi and Z. Q. Xu, {\it Non-homogeneous stochastic LQ control with regime switching and random coefficients}, Math. Control Relat. F.,
\textbf{14} no.2 (2024), pp.~671--694.


\bibitem {Xu arXiv2}
Y. Hu, X. M. Shi and Z. Q. Xu, {\it Optimal consumption-investment with coupled constraints
on consumption and investment strategies in a regime switching market with random coefficients}, (2022), arXiv:2211.05291.

\bibitem {Hu Peng}
Y. Hu and S. G. Peng, {\it On the comparison theorem for multidimensional BSDEs},
C. R. Acad. Sci. Paris, Ser. I, \textbf{343} no.2 (2006), pp.~135--140.

\bibitem {Hu Tang}
Y. Hu and S. J. Tang, {\it Multi-dimensional backward stochastic differential equations of diagonally quadratic generators},
Stochastic Process. Appl., \textbf{126} no.4 (2016), pp.~1066--1086.

\bibitem {Hu Zhou 2}
Y. Hu and X. Y. Zhou, {\it Indefinite stochastic Riccati equations},
SIAM J. Control Optim., \textbf{42} no.1 (2003), pp.~123--137.

\bibitem {Hu Zhou}
Y. Hu and X. Y. Zhou, {\it Constrained stochastic LQ control with random coefficients, and application to portfolio selection}
SIAM J. Control Optim., \textbf{44} no.2 (2005), pp.~444--466.



\bibitem {Isaacs}
R. Isaacs, {\it Differential games}, Wiley, New York, 1965.

\bibitem {Kazamaki}
N. Kazamaki, {\it Continuous exponential martingales and BMO}, Springer, 2016.

\bibitem {Kohlmann Tang}
M. Kohlmann and S. J. Tang, {\it Global adapted solution of one-dimensional backward stochastic
Riccati equations, with application to the mean-variance hedging},
Stochastic Process. Appl., \textbf{97} no.2 (2002), pp.~255--288.

\bibitem {Lv}
S. Y. Lv, {\it Two-player zero-sum stochastic differential games with regime switching},
Automatica, \textbf{114} (2020), 108819.


\bibitem {Moon}
J. Moon, {\it A sufficient condition for linear-quadratic stochastic zero-sum
differential games for Markov jump systems},
IEEE Trans. Automat. Contr., \textbf{64} (2019), pp.~1619--1626.

\bibitem {Moon2}
J. Moon, {\it A feedback Nash equilibrium for affine-quadratic zero-sum stochastic differential games with random coefficients},
IEEE Contr. Syst. Lett, \textbf{4} no.4 (2020), pp.~868--873.


\bibitem {Qian Zhou}
Z. M. Qian and X. Y. Zhou, {\it Existence of solutions to a class of indefinite stochastic Riccati equations},
SIAM J. Control Optim., \textbf{51} no.1 (2013), pp.~221--229.

\bibitem {Sion}
M. Sion, {\it On general minimax theorems}, Pac. J. Math., \textbf{8} no.1 (1958), pp.~171--176.

\bibitem {Sun}
J. R. Sun and H. X. Wang, {\it Linear-quadratic optimal control for backward stochastic differential equation with random coefficients},
ESAIM Control Optim. Calc. Var., \textbf{27} (2021), 46.

\bibitem {Tang SMP}
S. J. Tang, {\it General linear quadratic optimal stochastic control problems with random coefficients:
linear stochastic Hamilton systems and backward stochastic Riccati equations},
SIAM J. Control Optim., \textbf{42} no.1 (2003), pp.~53--75.

\bibitem {Tang DPP}
S. J. Tang, {\it Dynamic programming for general linear quadratic optimal stochastic control with random coefficients},
SIAM J. Control Optim., \textbf{53} no.2 (2015), pp.~1082--1106.

\bibitem{WLXZ23}
J. Q. Wen, X. Li, J. Xiong and X. Zhang,
{\it Stochastic linear-quadratic optimal control problems with random coefficients and Markovian regime switching system},
SIAM J. Control Optim., \textbf{61} no.2 (2023), pp.~949--979.

\bibitem {Yong Zhou}
J. M. Yong and X. Y. Zhou, {\it Stochastic controls: Hamiltonian systems and HJB equations}, Springer-Verlag, New York, USA, 1999.

\bibitem {Yu}
Z. Y. Yu, {\it An optimal feedback control-strategy pair for zero-sum linear-quadratic stochastic differential game: The Riccati equation approach}, SIAM J. Control Optim., \textbf{53} no.4 (2015), pp.~2141--2167.

\end{thebibliography}
\end{document}